\documentclass[11pt]{amsart}
\usepackage{graphicx}
\usepackage{amsmath,amssymb,xy,array}
\usepackage[T1]{fontenc}
\usepackage{amsfonts}
\usepackage{mathabx}
\usepackage{amsmath}
\usepackage{amssymb}
\usepackage{amsthm}
\usepackage{graphicx}
\usepackage{enumitem}
\usepackage{xcolor}
\usepackage[utf8]{inputenc}
\usepackage{hyperref}
\usepackage{mathrsfs}
\usepackage{comment}
\usepackage[toc,page]{appendix}
\usepackage{fancyhdr}
\usepackage{tikz}
\usepackage{tikz-cd}
\usepackage{longtable}
\usepackage{geometry}
\usepackage{textcomp}
\usetikzlibrary{trees}
\usetikzlibrary{arrows}
\usepackage{multirow}
\usepackage{youngtab}
\usepackage{mathdots}
\usetikzlibrary{arrows.meta}

\definecolor{yqyqyq}{rgb}{0.5019607843137255,0.5019607843137255,0.5019607843137255}\definecolor{uuuuuu}{rgb}{0.26666666666666666,0.26666666666666666,0.26666666666666666}
\definecolor{uququq}{rgb}{0.25098039215686274,0.25098039215686274,0.25098039215686274}
\definecolor{wwwwww}{rgb}{0.4,0.4,0.4}
\definecolor{uuuuuu}{rgb}{0.26666666666666666,0.26666666666666666,0.26666666666666666}

\setlist[itemize]{leftmargin=6mm}

\renewcommand{\P}{\mathbb P}

\newcommand{\Aut}{\operatorname{Aut}}

\DeclareMathOperator{\Cl}{Cl}

\DeclareMathOperator{\lin}{lin}

\DeclareMathOperator{\Hilb}{Hilb}
\DeclareMathOperator{\Chow}{Chow}

\DeclareMathOperator{\Hom}{Hom}

\DeclareMathOperator{\Exc}{Exc}

\DeclareMathOperator{\Jn}{Join}

\DeclareMathOperator{\Eff}{Eff}
\DeclareMathOperator{\Nef}{Nef}
\DeclareMathOperator{\Mov}{Mov}
\DeclareMathOperator{\Pic}{Pic}
\DeclareMathOperator{\rank}{rank}

\renewcommand{\sec}{\mathbb{S}ec}

\DeclareMathOperator{\Cox}{Cox}

\renewcommand{\P}{\mathbb{P}}

\newcommand{\mmn}{\overline{M}_{0,0}(\mathbb{P}^n,2)}

\newcommand{\mmnm}{\overline{M}_{0,0}(\mathbb{P}^n \times \mathbb{P}^m,(1,1))}
\newcommand{\mmgu}{\overline{M}_{0,0}(\mathbb{G}(1,n),2)}

\newtheorem{thm}{Theorem}[section]

\newtheorem{Lemma}[thm]{Lemma}
\newtheorem{Proposition}[thm]{Proposition}

\newtheorem{Corollary}[thm]{Corollary}

\theoremstyle{definition}

\newtheorem{Definition}[thm]{Definition}
\newtheorem{Remark}[thm]{Remark}

\newtheorem{Notation}[thm]{Notation}

\hypersetup{pdfpagemode=UseNone}
\hypersetup{pdfstartview=FitH}

\graphicspath{{immagini/}}
\usepackage{pgf,tikz,pgfplots}
\pgfplotsset{compat=1.15}
\usepackage{mathrsfs}
\usetikzlibrary{arrows}

\begin{document}

\title{Complete singular collineations and quadrics}

\author[Alex Casarotti]{Alex Casarotti}
\address{\sc Alex Casarotti\\ Dipartimento di Matematica, Universit\`a di Trento, Via Sommarive 14, 38123 Trento, Italy}
\email{alex.casarotti@unitn.it}

\author[Elsa Corniani]{Elsa Corniani}
\address{\sc Elsa Corniani\\ Dipartimento di Matematica e Informatica, Universit\`a di Ferrara, Via Machiavelli 30, 44121 Ferrara, Italy}
\email{elsa.corniani@unife.it}

\author[Alex Massarenti]{Alex Massarenti}
\address{\sc Alex Massarenti\\ Dipartimento di Matematica e Informatica, Universit\`a di Ferrara, Via Machiavelli 30, 44121 Ferrara, Italy}
\email{alex.massarenti@unife.it}

\date{\today}

\begin{abstract}
We construct wonderful compactifications of the spaces of linear maps, and symmetric linear maps of a given rank as blow-ups of secant varieties of Segre and Veronese varieties. Furthermore, we investigate their birational geometry and their relations with some spaces of degree two stable maps. 
\end{abstract}

\maketitle

\setcounter{tocdepth}{1}

\tableofcontents

\section{Introduction}

We construct the wonderful compactification of the space of linear maps of rank $h$, between two vector spaces of dimensions $n+1$ and $m+1$, as a sequence of blow-ups of secant varieties of Segre varieties. This generalizes a construction, due to I. Vainsencher, for complete collineations that is maps of maximal rank \cite[Theorem 1]{Va84}.

Complete collineations have been widely studied from the algebraic, enumerative and birational viewpoint since the 19th-century \cite{Ch64}, \cite{Gi03}, \cite{Hi75}, \cite{Hi77}, \cite{Sc86}, \cite{Se84}, \cite{Se48}, \cite{Se51}, \cite{Se52}, \cite{Ty56}, \cite{Va82}, \cite{Va84}, \cite{TK88}, \cite{LLT89}, \cite{Tha99}, \cite{Ce15}, \cite{Ma18a}, \cite{Ma18b}.

Spaces of complete collineations are examples of wonderful compactifications. The \textit{wonderful compactification} of a symmetric space was introduced by C. De Concini and C. Procesi in \cite{DP83}. Later on, D. Luna gave a more general definition of wonderful variety and then he proved that, according to his definition, all wonderful varieties are spherical \cite{Lu96}. 

Let $\mathscr{G}$ be a reductive group, and $\mathscr{B}\subset\mathscr{G}$ a Borel subgroup. A \textit{spherical variety} is a variety admitting an action of $\mathscr{G}$ with an open dense $\mathscr{B}$-orbit. For \textit{wonderful varieties} we require in addition the existence of an open orbit whose complementary set is a simple normal crossing divisor, $E_1\cup\dots\cup E_r$, where the $E_i$ are the $\mathscr{G}$-invariant prime divisors in the variety $X$.

Let $\mathcal{S}^{n,m}$ be the image of the Segre embedding $\mathbb{P}^n\times\mathbb{P}^m\rightarrow\mathbb{P}^N$, and $\sec_h(\mathcal{S}^{n,m})$ the $h$-secant variety of $\mathcal{S}^{n,m}$, that is the subvariety of $\mathbb{P}^N$ obtained as the closure of the union of all $(h-1)$-planes spanned by $h$ general points of $\mathcal{S}^{n,m}$. We summarize the main results in Theorem \ref{comp} and Propositions \ref{picrank}, \ref{effnef}.
\begin{thm}\label{A}
Consider the following sequence of blow-ups 
$$\mathcal{C}(n,m,h):=\sec_{h}^{(h-1)}(\mathcal{S}^{n,m})\rightarrow \sec_{h}^{(h-2)}(\mathcal{S}^{n,m})\rightarrow\dots\rightarrow \sec_{h}^{(1)}(\mathcal{S}^{n,m})\rightarrow \sec_{h}^{(0)}(\mathcal{S}^{n,m}):=\sec_{h}(\mathcal{S}^{n,m})$$
where $\sec_{h}^{(k)}(\mathcal{S}^{n,m})\rightarrow \sec_{h}^{(k-1)}(\mathcal{S}^{n,m})$ is the blow-up of $\sec_{h}^{(k-1)}(\mathcal{S}^{n,m})$ along the strict transform of $\sec_k(\mathcal{S}^{n,m})$ for $k = 1,\dots, h-1$. Denote by $E_k\subset\mathcal{C}(n,m,h)$ the exceptional divisor over $\sec_k(\mathcal{S}^{n,m})$ for $k=1,\dots,h-1$.

The $(SL(n+1)\times SL(m+1))$-action
$$\begin{array}{cccc}
(SL(n+1)\times SL(m+1))\times \mathbb{P}^N & \longrightarrow & \mathbb{P}^N\\
((A,B),Z) & \longmapsto & AZB^{t}
\end{array}$$
induces an $(SL(n+1)\times SL(m+1))$-action on $\mathcal{C}(n,m,h)$, and $\mathcal{C}(n,m,h)$ is wonderful.  

Assume that $h < n+1$ and fix homogeneous coordinates $[z_{0,0}:\dots:z_{n,n}]$ on $\mathbb{P}^N$. For $i=1,\dots,h$ we define the divisors $D_i^{\mathcal{C}}$ as the strict transforms in $\mathcal{C}(n,m,h)$ of the divisor given by the intersection of  
$$\det \begin{pmatrix}
z_{0,0} & \dots & z_{0,i-1}\\
\vdots & \ddots & \vdots \\
z_{i-1,0} & \dots & z_{i-1,i-1}\\
\end{pmatrix}=0$$
with $\mathcal{C}(n,m,h)$. The divisor $D_h^\mathcal{C}$ in $\mathcal{C}(n,m,h)$ has two irreducible components $H_1^{\mathcal{C}}, H_2^{\mathcal{C}}$, and the Picard rank of $\mathcal{C}(n,m,h)$ is $\rho(\mathcal{C}(n,m,h)) = h+1$. Moreover, the effective cone $\Eff(\mathcal{C}(n,m,h))$ is generated by $E_1,\dots,E_{h-1},H_1^{\mathcal{C}}, H_2^{\mathcal{C}}$ and the nef cone $\Nef(\mathcal{C}(n,m,h))$ is generated by $D_1^{\mathcal{C}},\dots,D_{h-1}^{\mathcal{C}},H_1^{\mathcal{C}}, H_2^{\mathcal{C}}$.
\end{thm}

In the case $h = n+1$ we present similar results. Furthermore, we extend the construction in Theorem \ref{A}, by replacing $\mathcal{S}^{n,m}$ with the Veronese variety $\mathcal{V}^{n}$, to the space $\mathcal{Q}(n,h)$ of rank $h$ symmetric complete collineations.  

Note that both $\sec_h(\mathcal{S}^{n,m})$ and $\sec_h(\mathcal{V}^{n})$ are singular, the wonderful varieties $\mathcal{C}(n,m,h)$ and $\mathcal{Q}(n,h)$ are examples of the process producing a wonderful compactification from a conical one in \cite{MP98}.

Spherical varieties are Mori dream spaces. Roughly, a \textit{Mori dream space} is a projective variety $X$ whose cone of effective divisors $\Eff(X)$ admits a well-behaved decomposition into convex sets, called Mori chamber decomposition, and these chambers are the nef cones of the birational models of $X$ \cite{HK00}.
  
In Propositions \ref{mcdq3} and \ref{mcd_C} we give a detailed description of the Mori chamber decompositions of $\mathcal{C}(n,m,h)$ and $\mathcal{Q}(n,h)$ when their Picard rank is at most three. Moreover, in Section \ref{sec4} we investigate the connection of $\mathcal{C}(n,m,h)$ and $\mathcal{Q}(n,h)$ with some Kontsevich spaces of degree two maps. 

Kontsevich moduli spaces are denoted by $\overline{M}_{g,n}(X,\beta)$ where $X$ is a projective scheme and $\beta\in H_2(X,\mathbb{Z})$ is the homology class of a curve in $X$. A point in $\overline{M}_{g,n}(X,\beta)$ corresponds to a holomorphic map $\alpha$ from an $n$-pointed genus $g$ curve $C$ to $X$ such that $\alpha_{*}([C])=\beta$. When $X$ is a projective space or a Grassmannians the class $\beta$ is completely determined by its degree, similarly when $X$ is the product of two projective spaces we identify the class $\beta$ with its the bidegree. By Propositions \ref{isoQ}, \ref{isoC}, \ref{mapgr}, \ref{aut_MG}, and Corollary \ref{aut_M} we have the following:
 
\begin{thm}\label{B}
There are isomorphisms 
$$\mathcal{C}(n,m,2)\xrightarrow{\sim} \mmnm$$
and 
$$\sec_3^{(1)}(\mathcal{V}^n)\xrightarrow{\sim} \mmn.$$
Furthermore, there is a $2$-to-$1$ morphism 
$$\overline{M}_{0,0}(\mathbb{G}(1,n),2)\rightarrow \sec_4^{(2)}(\mathcal{V}^n).$$
For the automorphism groups we have that 
$$
\Aut(\overline{M}_{0,0}(\mathbb{P}^n\times\mathbb{P}^m,(1,1))) \cong
\left\lbrace\begin{array}{ll}
PGL(n+1)\times PGL(m+1) & \text{if n}< \textit{m};\\ 
S_2 \ltimes (PGL(n+1)\times PGL(n+1)) & \text{if n} = \textit{m}\geq 2;
\end{array}\right.
$$
and $\Aut(\overline{M}_{0,0}(\mathbb{P}^1\times\mathbb{P}^1,(1,1)))\cong PGL(4)$.

Furthermore, $\Aut(\overline{M}_{0,0}(\mathbb{P}^n,2))\cong PGL(n+1)$ for $n\geq 3$, $\Aut(\overline{M}_{0,0}(\mathbb{P}^2,2))\cong PGL(3)\rtimes S_2$, and $\Aut(\overline{M}_{0,0}(\mathbb{P}^1,2))\cong PGL(3)$. 

Finally,
$$
\Aut(\mmgu) \cong
\left\lbrace\begin{array}{ll}
S_2 \ltimes PGL(n+1) & \text{if n}>  3;\\ 
S_2 \ltimes (S_2 \ltimes PGL(n+1)) & \text{if n} = 3.
\end{array}\right.
$$
\end{thm}
 
The Mori theory of the spaces $\overline{M}_{0,n}(X,\beta)$, especially when the target variety is a projective space or a Grassmannian, has been widely investigated in a series of papers \cite{CS06}, \cite{Ch08}, \cite{CHS08}, \cite{CHS09}, \cite{CC10}, \cite{CC11}, \cite{CM17}. As an application of Theorem \ref{B} we recover some of these results in Propositions \ref{mcd_pn2}, \ref{isoC}, and Remark \ref{bir_n2}. In particular, Theorem \ref{B} gives an explicit description of the birational contraction of $\overline{M}_{0,0}(\mathbb{P}^n,2)$ in \cite[Theorem 1.2]{CHS09} as the blow-down $\sec_3^{(1)}(\mathcal{V}^n)\rightarrow \sec_3(\mathcal{V}^n)$.

\subsection*{Organization of the paper} Throughout the paper we work over an algebraically closed field $K$ of characteristic zero. In Section \ref{sec1}, we construct the spaces of complete singular collineations and quadrics, $\mathcal{C}(n,m,h)$ and $\mathcal{Q}(n,h)$. In Section \ref{pic}, we study their Picard rank, their effective and nef cones, and compute the Mori chamber decomposition of $\mathcal{C}(n,m,2)$ and $\mathcal{Q}(n,3)$. Finally, in Section \ref{sec4} we investigate the relation of the space of complete singular collineations and quadrics with Kontsevich moduli spaces of conics. 

\subsection*{Acknowledgments}
The first and the third named authors are members of the Gruppo Nazionale per le Strutture Algebriche, Geometriche e le loro Applicazioni of the Istituto Nazionale di Alta Matematica "F. Severi" (GNSAGA-INDAM). The first named author is supported by Fondo PRIN-MIUR  "Moduli Theory and Birational Classification" 2017. We thank the referee for many helpful comments that allowed us to improve the paper.

\section{Complete rank $h$ collineations}\label{sec1}
Let $V,W$ be $K$-vector spaces of dimension respectively $n+1$ and $m+1$ with $n\leq m$, and let $\mathbb{P}^N$ with $N = nm+n+m$ be the projective space parametrizing collineations from $V$ to $W$ that is non-zero linear maps $V\rightarrow W$ up to a scalar multiple. 

The line bundle $\mathcal{O}_{\mathbb{P}^n\times \mathbb{P}^m}(1,1)=\mathcal{O}_{\mathbb{P}(V)}(1)\boxtimes\mathcal{O}_{\mathbb{P}(W)}(1)$
induces an embedding
$$
\begin{array}{cccc}
\sigma:
&\mathbb{P}(V)\times\mathbb{P}(W)& \longrightarrow & \mathbb{P}(V\otimes W)
=\mathbb{P}^{N}\\
      & (\left[u\right],\left[v\right]) & \longmapsto & [u\otimes v].
\end{array}
$$ 
The image $\mathcal{S}^{n,m} = \sigma(\mathbb{P}^n\times \mathbb{P}^m) \subset \mathbb{P}^{N}$ is the \textit{Segre variety}. Let $[x_0,\dots, x_n],[y_0,\dots,y_m]$ be homogeneous coordinates respectively on $\mathbb{P}^n$ and $\mathbb{P}^m$. Then the morphism $\sigma$ can be written as
$$\sigma([x_0,\dots, x_n],[y_0,\dots,y_m]) = [x_0y_0:\dots:x_0y_m:x_1y_0:\dots :x_ny_m].$$
We will denote by $[z_{0,0}:\dots :z_{n,m}]$ the homogeneous coordinates on $\mathbb{P}^N$, where $z_{i,j}$ corresponds to the product $x_iy_j$.

A point $p\in \mathbb{P}^N = \mathbb{P}(\Hom(W,V))$ can be represented by an $(n+1)\times (m+1)$ matrix $Z$. The Segre variety $\mathcal{S}^{n,m}$ is the locus of rank one matrices. More generally, $p\in \sec_h(\mathcal{S}^{n,m})$ if and only if $Z$ can be written as a linear combination of $h$ rank one matrices that is if and only if $\rank(Z)\leq h$. If $p = [z_{0,0}:\cdots:z_{n,m}]$ then we may write
\stepcounter{thm}
\begin{equation}\label{matrix}
Z = \left(
\begin{array}{ccc}
z_{0,0} & \dots & z_{0,m}\\ 
\vdots & \ddots & \vdots\\ 
z_{n,0} & \dots & z_{n,m}
\end{array}\right).
\end{equation}
Therefore, the ideal of $\sec_h(\mathcal{S}^{n,m})$ is generated by the $(h+1)\times (h+1)$ minors of $Z$.

\subsection{Spherical and Wonderful varieties}
Let $X$ be a normal projective $\mathbb{Q}$-factorial variety. We denote by $N^1(X)$ the real vector space of $\mathbb{R}$-Cartier divisors modulo numerical equivalence. 
The \emph{nef cone} of $X$ is the closed convex cone $\Nef(X)\subset N^1(X)$ generated by classes of nef divisors. 

The stable base locus $\textbf{B}(D)$ of a $\mathbb{Q}$-divisor $D$ is the set-theoretic intersection of the base loci of the complete linear systems $|sD|$ for all positive integers $s$ such that $sD$ is integral
\stepcounter{thm}
\begin{equation}\label{sbl}
\textbf{B}(D) = \bigcap_{s > 0}B(sD).
\end{equation}
The \emph{movable cone} of $X$ is the convex cone $\Mov(X)\subset N^1(X)$ generated by classes of 
\emph{movable divisors}. These are Cartier divisors whose stable base locus has codimension at least two in $X$.
The \emph{effective cone} of $X$ is the convex cone $\Eff(X)\subset N^1(X)$ generated by classes of 
\emph{effective divisors}. We have inclusions $\Nef(X)\ \subset \ \overline{\Mov(X)}\ \subset \ \overline{\Eff(X)}$. We refer to \cite[Chapter 1]{De01} for a comprehensive treatment of these topics. 

\begin{Definition}
A \textit{spherical variety} is a normal variety $X$ together with an action of a connected reductive affine algebraic group $\mathscr{G}$, a Borel subgroup $\mathscr{B}\subset \mathscr{G}$, and a base point $x_0\in X$ such that the $\mathscr{B}$-orbit of $x_0$ in $X$ is a dense open subset of $X$. 

Let $(X,\mathscr{G},\mathscr{B},x_0)$ be a spherical variety. We distinguish two types of $\mathscr{B}$-invariant prime divisors: a \textit{boundary divisor} of $X$ is a $\mathscr{G}$-invariant prime divisor on $X$, a \textit{color} of $X$ is a $\mathscr{B}$-invariant prime divisor that is not $\mathscr{G}$-invariant. We will denote by $\mathcal{B}(X)$ and $\mathcal{C}(X)$ respectively the set of boundary divisors and colors of $X$.
\end{Definition}

\begin{Definition}
A \textit{wonderful variety} is a smooth projective variety $X$ with the action of a semi-simple simply connected group $\mathscr{G}$ such that:
\begin{itemize}
\item[-] there is a point $x_0\in X$ with open $\mathscr{G}$ orbit and such that the complement $X\setminus \mathscr{G}\cdot x_0$ is a union of prime divisors $E_1,\cdots, E_r$ having simple normal crossing;
\item[-] the closures of the $\mathscr{G}$-orbits in $X$ are the intersections $\bigcap_{i\in I}E_i$ where $I$ is a subset of $\{1,\dots, r\}$.
\end{itemize} 
\end{Definition}  

As proven by D. Luna in \cite{Lu96} wonderful varieties are in particular spherical. 

\subsection{Complete singular forms}
For $n = m$, let $\mathbb{P}^{N_{+}}\subset\mathbb{P}^N$ be the subspace of symmetric matrices. Then $\sec_h(\mathcal{S}^{n,m}) \cap \mathbb{P}^{N_{+}} = \sec_h(\mathcal{V})$ for any $h\geq 1$, where $\mathcal{V}^{n}\subset\mathbb{P}^{N_{+}}$ is the image of the degree two Veronese embedding of $\mathbb{P}^n$. 

\begin{Definition}\label{def1}
The space of \textit{complete rank $h$ collineations} is the variety $\mathcal{C}(n,m,h)$ obtained by blowing-up $\sec_h(\mathcal{S}^{n,m})$ along the strict transforms of the secant varieties $\sec_k(\mathcal{S}^{n,m})$ for $k< h$ in order of increasing dimension. When $n = m$ we will denote $\mathcal{C}(n,n,h)$ simply by $\mathcal{C}(n,h)$. Furthermore, we will denote by $E_1,\dots,E_{h-1}$ the exceptional divisors.

Similarly, for $n = m$ the space of \textit{complete rank $h$ quadrics} is the variety $\mathcal{Q}(n,h)$ obtained by blowing-up $\sec_h(\mathcal{V}^{n})$ along the strict transforms of the secant varieties $\sec_k(\mathcal{V}^{n})$ for $k< h$ in order of increasing dimension. We will denote by $E_1^{\mathcal{Q}},\dots,E_{h-1}^{\mathcal{Q}}$ its exceptional divisors.
\end{Definition}

\begin{Remark}
The case $\mathcal{C}(n,m,n+1)$ and $\mathcal{Q}(n,n+1)$ are respectively the space of complete collineations from $V$ to $W$ and the space of complete quadrics of $V$. By \cite[Theorem 1]{Va84} and \cite[Theorem 6.3]{Va82} they are wonderful varieties and their birational geometry has been studied in \cite{Ma18a}.
\end{Remark}

\begin{Notation}
For $k \le h$, we will denote by $\sec_h^{(k)}(\mathcal{S}^{n,m})$ the blow-up of $\sec_h(\mathcal{S}^{n,m})$ along the strict transforms of the secant varieties $\sec_i(\mathcal{S}^{n,m})$ for $i=1, \dots, k$, and by $\sec_h^{(k)}(\mathcal{V}^{n})$ the blow-up of $\sec_h(\mathcal{V}^{n})$ along the strict transforms of the secant varieties $\sec_i(\mathcal{V}^{n})$ for $i=1, \dots, k$.
\end{Notation}

Note that there is an embedding 
\stepcounter{thm}
\begin{equation}\label{emb}
i:\mathcal{Q}(n,h)\hookrightarrow\mathcal{C}(n,h).
\end{equation}

The following $(SL(n+1)\times SL(m+1))$-action
\stepcounter{thm} 
\begin{equation}\label{actcol}
\begin{array}{cccc}
(SL(n+1)\times SL(m+1))\times \mathbb{P}^N & \longrightarrow & \mathbb{P}^N\\
((A,B),Z) & \longmapsto & AZB^{t}
\end{array}
\end{equation}
induces an $(SL(n+1)\times SL(m+1))$-action on $\mathcal{C}(n,m,h)$. Similarly, when $n = m$ the $SL(n+1)$-action 
\stepcounter{thm}
\begin{equation}\label{actquad}
\begin{array}{cccc}
SL(n+1)\times \mathbb{P}^{N_{+}} & \longrightarrow & \mathbb{P}^{N_{+}}\\
(A,Z) & \longmapsto & AZA^{t}
\end{array}
\end{equation}
induces an $SL(n+1)$-action on $\mathcal{Q}(n,h)$.

\begin{Remark}\label{sec_ver}
Since $\sec_h(\mathcal{S}^{n,m})$ can be identified with the variety of $(n+1)\times (m+1)$ matrices modulo scalar of rank at most $h$, \cite[Example 12.1]{Ha95}, \cite[Proposition 12(a)]{HT84} give
$$\dim(\sec_h(\mathcal{S}^{n,m})) =h(m+n+2-h)-1,\quad \deg(\sec_h(\mathcal{S}^{n,m})) = \prod_{i=0}^{n-h}\frac{\binom{m+1+i}{n-i}}{\binom{m+1-h+i}{n-h-i}}.$$
Similarly, $\sec_h(\mathcal{V}^{n})$ identifies with the variety parametrizing $(n+1)\times (n+1)$ symmetric matrices modulo scalar of rank at most $h$ and
$$\dim(\sec_h(\mathcal{V}^{n})) = \frac{2nh-h^2+3h-2}{2}, \quad \deg(\sec_h(\mathcal{V}^n)) = \prod_{i=0}^{n-h}\frac{\binom{n+1+i}{n+1-h-i}}{\binom{2i+1}{i}}.$$
\end{Remark}

\begin{Proposition}\label{tcones}
The tangent cone of $\sec_h(\mathcal{S}^{n,m})$ at a point $p\in \sec_k(\mathcal{S}^{n,m})\setminus\sec_{k-1}(\mathcal{S}^{n,m})$ for $k\leq h$ is a cone with vertex of dimension $nm+n+m-(m+1-k)(n+1-k)$ over $\sec_{h-k}(\mathcal{S}^{n-k,m-k})$. 

The tangent cone of $\sec_h(\mathcal{V}^n)$ at a point $p\in \sec_k(\mathcal{V}^n)\setminus\sec_{k-1}(\mathcal{V}^n)$ for $k\leq h$ is a cone with vertex of dimension $\binom{n+2}{2}-1-\frac{(n-k+1)(n-k+2)}{2}$ over $\sec_{h-k}(\mathcal{V}^{n-k})$. 
\end{Proposition}
\begin{proof}
We compute the tangent cones of $\sec_h(\mathcal{S}^{n,m})$. The symmetric case can be worked out similarly. It is enough to compute the tangent cone of $\sec_h(\mathcal{S}^{n,m})$ at 
$$
p_k = \left(
\begin{array}{cc}
I_{k,k} & 0_{k,m+1-k} \\ 
0_{n+1-k,k} & 0_{n+1-k,m+1-k}
\end{array}
\right)
$$ 
where $I_{k,k}$ is the $k\times k$ identity matrix. Consider the affine chart $z_{0,0}\neq 0$ and the change of coordinates $z_{i,i}\mapsto z_{i,i}-1$ for $i = 1,\dots,k-1$, $z_{i,j}\mapsto z_{i,j}$ otherwise. Then the matrix $Z$ in (\ref{matrix}) takes the following form
$$
\left(
\begin{array}{ccccccc}
1 & z_{0,1} & \hdots & z_{0,k-1} & z_{0,k} & \hdots & z_{0,m} \\ 
z_{1,0} & z_{1,1}-1 & \hdots & z_{1,k-1} & z_{1,k} & \hdots & z_{1,m} \\ 
\vdots & \vdots & \ddots & \vdots & \vdots & \ddots & \vdots \\ 
z_{k-1,0} & z_{k-1,1} & \hdots & z_{k-1,k-1}-1 & z_{k-1,k} & \hdots & z_{k-1,m} \\ 
z_{k,0} & z_{k,1} & \hdots & z_{k,k-1} & z_{k,k} & \hdots & z_{k,m} \\ 
\vdots & \vdots & \ddots & \vdots & \vdots & \ddots & \vdots \\ 
z_{n,0} & z_{n,1} & \hdots & z_{n,k-1} & z_{n,k} & \hdots & z_{n,m}
\end{array} 
\right).
$$
Recall that $\sec_h(\mathcal{S}^{n,m})\subseteq\mathbb{P}^N$ is cut out by the $(h+1)\times (h+1)$ minors of $Z$. Now, the lowest degree terms of these minors are given by the $(h+1-k)\times (h+1-k)$ minors of the following matrix
$$
\left(
\begin{array}{ccc}
z_{k,k} & \hdots & z_{k,m}\\ 
\vdots & \ddots & \vdots \\ 
z_{n,k} & \hdots & z_{n,m}
\end{array} 
\right).
$$
Therefore, the tangent cone $TC_{p_k}\sec_h(\mathcal{S}^{n,m})$ is contained in the cone $C$ over $\sec_{h-k}(\mathcal{S}^{n-k,m-k})$ with vertex the linear subspace of $\mathbb{P}^N$ given by $\{z_{k,k} =\dots = z_{k,m} = z_{k+1,k} = \dots = z_{k+1,m} = \dots = z_{n,k}= \dots =z_{n,m}=0\}$. Finally, by Remark \ref{sec_ver} we conclude that $TC_{p_k}\sec_h(\mathcal{S}^{n,m}) = C$. 
\end{proof}

We will need the following result on fibrations with smooth fibers on a smooth base.

\begin{Proposition}\label{smooth_fib}
Let $f:X\rightarrow Y$ be a surjective morphism of varieties over an algebraically closed field with equidimensional smooth fibers. If $Y$ is smooth then $X$ is smooth as well. 
\end{Proposition}
\begin{proof}
By \cite[Theorem 3.3.27]{Sch99} the morphism $f:X\rightarrow Y$ is flat. Finally, since all the fibers of $f:X\rightarrow Y$ are smooth and of the same dimension \cite[Theorem 3', Chapter III, Section 10]{Mum99} yields that $X$ is smooth. 
\end{proof}

\begin{thm}\label{comp}
The variety $\mathcal{C}(n,m,h)$ is smooth and the divisors $E_1,\dots,E_{h-1}$ are smooth and intersect transversally. The closures of the orbits of the $SL(n+1) \times SL(m+1)$-action on $\mathcal{C}(n,m,h)$ induced by (\ref{actcol}) are given by all the possible intersections of $E_1,\dots,E_{h-1}$ and $\mathcal{C}(n,m,h)$. Furthermore, the analogous statements hold for $\mathcal{Q}(n,h)$. Hence $\mathcal{C}(n,m,h)$ and $\mathcal{Q}(n,h)$ are wonderful.
\end{thm}
\begin{proof}
We will proceed as follows. For $h = 1$ we will prove the statement for any $n$ and $m$. Then we will prove that if for $h<j$ the statement holds for any $n$ and $m$ then it also holds for $h = j$ and any $n$ and $m$. This will prove the statement for any $n$,$m$ and $h = 0,\dots, n+1$. 

For $h = 1$ we have $\mathcal{C}(n,m,1) = \mathcal{S}^{n,m}$. Hence, the statements holds for any $n$ and $m$. Assume that for any $h < j$ the statement holds for any $n$ and $m$ and consider $\mathcal{C}(n,m,j)$.

In order to understand the geometry of our construction it is more useful to focus on a specific case. For instance take $n = m = 3$. We have $\mathcal{S}^{3,3}\subset \sec_2(\mathcal{S}^{3,3})\subset \sec_3(\mathcal{S}^{3,3})\subset\mathbb{P}^{15}$. Let $X_1$ be the blow-up of $\mathbb{P}^{15}$ along $\mathcal{S}^{3,3}$ with exceptional divisor $\overline{E}_1$. Then $\overline{E}_1$ is a $\mathbb{P}^{8}$-bundle over $\mathcal{S}^{3,3}$. The strict transform $\sec_2^{(1)}(\mathcal{S}^{3,3})$ intersects the fiber $\overline{E}_{1,p}$ of $\overline{E}_1$ over a point $p \in \mathcal{S}^{3,3}$ along the base of the tangent cone of $\sec_2(\mathcal{S}^{3,3})$ at $p$ which by Proposition \ref{tcones} is $\mathcal{S}^{2,2}$. Similarly, $\sec_3(\mathcal{S}^{3,3})$ intersects $\overline{E}_{1,p}$ along $\sec_2(\mathcal{S}^{2,2})$. Hence, the fibers of $E_1\rightarrow\mathcal{S}^{3,3}$ are secant varieties $\sec_2(\mathcal{S}^{2,2})$. Now, let $X_2$ be the blow-up of $X_1$ along $\sec_2^{(1)}(\mathcal{S}^{3,3})$ with exceptional divisor $\overline{E}_2$. Then $\overline{E}_2\rightarrow\sec_2^{(1)}(\mathcal{S}^{3,3})$ is a $\mathbb{P}^3$-bundle. Fix a point $p\in \sec_2^{(1)}(\mathcal{S}^{3,3})\setminus (E_1\cap \sec_2^{(1)}(\mathcal{S}^{3,3}))$. By Proposition \ref{tcones} $\sec_3^{(2)}(\mathcal{S}^{3,3})$ intersects $\overline{E}_{2,p}$ along $\mathcal{S}^{1,1}$. If $p\in \sec_2^{(1)}(\mathcal{S}^{3,3})\cap E_1$ then the projective tangent cone of $\sec_3^{(1)}(\mathcal{S}^{3,3})$ at $p$ coincides with the projective tangent cone of $\sec_3^{(1)}(\mathcal{S}^{3,3})\cap E_{1,p} = \sec_2(\mathcal{S}^{2,2})$ at $p\in\sec_2^{(1)}(\mathcal{S}^{3,3})\cap E_{1,p} = \mathcal{S}^{2,2}$ which in turn by Proposition \ref{tcones} is $\mathcal{S}^{1,1}$. Hence, the fibers of $E_2\rightarrow\sec_2^{(1)}(\mathcal{S}^{3,3})$ are isomorphic to $\mathcal{S}^{1,1}$. Summing up after the two blow-ups the fibers of $E_{1}\rightarrow\mathcal{S}^{3,3}$ are isomorphic to $\mathcal{C}(2,2,2)$, that is the blow-up of $\sec_2(\mathcal{S}^{2,2})$ along $\mathcal{S}^{2,2}$, and the fibers of $E_2\rightarrow\sec_2^{(1)}(\mathcal{S}^{3,3})$ are isomorphic to $\mathcal{C}(1,1,1)$ that is $\mathcal{S}^{1,1}$.   

Arguing in the same way we see that for any $i=1, \dots, j-1$, Proposition \ref{tcones} gives a fibration $E_i \rightarrow \sec_i^{(i-1)}(\mathcal{S}^{n,m})=\mathcal{C}(n,m,i)$ whose fibers are isomorphic to $\sec_{j-i}^{(j-i-1)}(\mathcal{S}^{n-i,m-i}) = \mathcal{C}(n-i,m-i,j-i)$. Then, by the induction hypothesis and Proposition \ref{smooth_fib} the exceptional divisors $E_1, \dots, E_{j-1}$ in $\mathcal{C}(n,m,j)$ are smooth.  Moreover, by Proposition \ref{tcones}, $\mathcal{C}(n,m,j)$ is smooth away from $E_1, \dots, E_{j-1}$ and for $i=1, \dots, j-1$ there is a fibration $\mathcal{C}(n,m,j)\cap E_i\rightarrow \mathcal{C}(n,m,i)$ whose fibers are isomorphic to $\mathcal{C}(n-i,m-i,j-i)$. Hence, by induction and Proposition \ref{smooth_fib} we get that $\mathcal{C}(n,m,j)\cap E_i$ is smooth and 
$$\dim(\mathcal{C}(n,m,j)\cap E_i) =i(n+m-i)-1+(j-i)(n-i+m-i-j+i)-1 = \dim(\mathcal{C}(n,m,j))-1.$$ 
So $\mathcal{C}(n,m,j)$ is smooth and the intersection $\mathcal{C}(n,m,j)\cap E_i$ is transversal for any $i=1, \dots, j-1$.

Now, consider an intersection of the following form $ E_{j_1}\cap\dots \cap E_{j_t}$. By Proposition \ref{tcones} the restriction of the blow-down morphism
$$E_{j_1}\cap\dots \cap E_{j_t}\rightarrow E_{j_1}\cap \dots \cap E_{j_{t-1}}\cap \mathcal{C}(n,m,j_t)$$
has fibers isomorphic to $\mathcal{C}(n-j_t,m-j_t,j-j_t)$. Again by the induction hypothesis and Proposition \ref{smooth_fib} $E_{j_1}\cap\dots \cap E_{j_t}$ is smooth of dimension 
$$(j-j_t)(n-j+m-j-j+j_t)-1+j_t(n+m-j_t)-1-(t-1) = \dim(\mathcal{C}(n,m,j))-t$$
and hence the intersection is transversal.

The claim about the orbit closures follows from \cite[Theorem 1]{Va84} and the fact that the $SL(n+1) \times SL(m+1)$ action on $\mathcal{C}(n,m,h)$ is given by the restriction of the action (\ref{actcol}) on the space of complete collineations. With an analogous proof we get the result for $\mathcal{Q}(n,h)$.
\end{proof}

\section{Divisors on $\mathcal{C}(n,m,h)$ and $\mathcal{Q}(n,h)$}\label{pic}
In the section we study the Picard groups and the cones of effective and nef divisors of the wonderful varieties introduces in Section \ref{sec1}. We will denote by $\mathcal{C}(n,m,h)^{o}$ and $\mathcal{Q}(n,h)^{o}$ the orbits of the matrix 
\stepcounter{thm}
\begin{equation}\label{matI}
J_h = \left(\begin{array}{cc}
I_{h,h} & 0 \\ 
0 & 0
\end{array} 
\right)
\end{equation}
where $I_{h,h}$ is the $h\times h$ identity matrix, under the actions (\ref{actcol}) and (\ref{actquad}) respectively.

\begin{Proposition}\label{picg}
The Picard groups of $\mathcal{C}(n,m,h)^{o}$ and $\mathcal{Q}(n,h)^{o}$ are given by
$$
\Pic(\mathcal{C}(n,m,h)^{o})\cong
\left\lbrace
\begin{array}{ll}
\mathbb{Z} & \text{if } h = n+1< m+1;\\ 
\mathbb{Z}\oplus\mathbb{Z} & \text{if } h<n+1;\\
\frac{\mathbb{Z}}{(n+1)\mathbb{Z}} & \text{if } h = n+1 = m+1;
\end{array}\right.
$$
and
$$
\Pic(\mathcal{Q}(n,h)^{o})\cong
\left\lbrace
\begin{array}{ll}
\mathbb{Z} & \text{if } h < n+1 \text{ is odd};\\ 
\frac{\mathbb{Z}}{2\mathbb{Z}}\oplus\mathbb{Z} & \text{if } h < n+1 \text{ is even};\\
\frac{\mathbb{Z}}{(n+1)\mathbb{Z}} & \text{if } h = n+1.
\end{array}\right.
$$
\end{Proposition}
\begin{proof}
Let $G_h$ be the stabilizer of the matrix $J_h$ in (\ref{matI}) under the action (\ref{actcol}). Since the Picard group and the character group of $SL(n+1)\times SL(m+1)$ are trivial \cite[Theorem 4.5.1.2]{ADHL15} yields that $\Pic(\mathcal{C}(n,m,h)^{o})$ is isomorphic to the character group $\mathbb{X}(G_h)$ of $G_h$. Write an element $(A,B)\in SL(n+1)\times SL(m+1)$ as 
\stepcounter{thm}
\begin{equation}\label{matdiv}
A = \left(\begin{array}{cc}
A_{h,h} & A_{h,n+1-h}\\ 
A_{n+1-h,h} & A_{n+1-h,n+1-h}
\end{array}\right),
\quad
B = \left(\begin{array}{cc}
B_{h,h} & B_{h,m+1-h}\\ 
B_{m+1-h,h} & B_{m+1-h,m+1-h}
\end{array}\right). 
\end{equation}
Then $(A,B)\in G_h$ if and only if $A_{n+1-h,h} = 0$, $B_{m+1-h,h} = 0$ and $A_{h,h}B_{h,h}^{T} = \lambda I_{h,h}$. Assume that $h < n+1$ and $h < m+1$. Then $\mathbb{X}(G_h)$ is generated by the characters 
$$d_{A_h} := \det(A_{h,h}), d_{B_h} := \det(B_{h,h}), d_{A_{n+1-h}} := \det(A_{n+1-h,n+1-h}), d_{B_{m+1-h}} := \det(B_{m+1-h,m+1-h}), \lambda$$ 
with the following relations
$$d_{A_h} + d_{A_{n+1-h}} = d_{B_h}+d_{B_{m+1-h}} =0, d_{A_h}+d_{B_h} = h\lambda.$$
Hence, $\mathbb{X}(G_h)$ is the free abelian group generated by $d_{A_h}$ and $\lambda$.

Now, assume that $h = n+1< m+1$. Then $d_{A_{n+1-h}} = 0$ and so $d_{A_h} = 0$. Therefore, $\mathbb{X}(G_h)$ is the free abelian group generated by $\lambda$. 

If $h = n+1 = m+1$ then $d_{A_{n+1-h}} = d_{B_{m+1-h}} = 0$. So $d_{A_h} = d_{B_h} = 0$, and hence $\mathbb{X}(G_h)$ is the abelian group generated by $\lambda$ with the relation $(n+1)\lambda = 0$.

Now, we consider the symmetric case. We will keep denoting by $G_h$ the stabilizer of the matrix $J_h$ in (\ref{matI}) under the action (\ref{actquad}). Write an element $A\in SL(n+1)$ as in (\ref{matdiv}). Then $A\in G_h$ if and only if $A_{n+1-h,h} = 0$ and $A_{h,h}A_{h,h}^{T} = \lambda I_{h,h}$. Therefore, $\mathbb{X}(G_h)$ is generated by 
$$d_{A_h} := \det(A_{h,h}), \lambda$$
with the relation $2d_{A_{h}}-h\lambda = 0$. 

Assume $h < n+1$. If $h = 2k+1$ then $(2,-h)\in \mathbb{Z}^{2}$ is primitive. Considering the basis $u = 2d_{A_{h}}-h\lambda, v = d_{A_{h}}-k\lambda$ of $\mathbb{Z}^2$ we get that $\mathbb{X}(G_h)\cong \mathbb{Z}^2/\left\langle u\right\rangle\cong \mathbb{Z}$. If $h = 2k$ then $(2,-h) = 2(1,-k)$, and considering the basis $u = 2d_{A_{h}}-k\lambda, v = \lambda$ of $\mathbb{Z}^2$ we get that $\mathbb{X}(G_h)\cong \mathbb{Z}^2/\left\langle 2u\right\rangle\cong \mathbb{Z}/2\mathbb{Z}\oplus\mathbb{Z}$. Finally, if $h = n+1$ we have $d_{A_{h}} = 0$, and hence $(n+1)\lambda = 0$. So $\mathbb{X}(G_h)\cong \mathbb{Z}/(n+1)\mathbb{Z}$.
\end{proof}

\begin{Proposition}\label{picrank}
The Picard rank of $\mathcal{C}(n,m,h)$ and $\mathcal{Q}(n,h)$ is given by
$$
\rho(\mathcal{C}(n,m,h))=
\left\lbrace
\begin{array}{ll}
h-1 & \text{if } h = n+1 = m+1;\\ 
h+1 & \text{if } h<n+1;\\
h & \text{if } h = n+1 < m+1;
\end{array}\right.
$$
and
$$
\rho(\mathcal{Q}(n,h))=
\left\lbrace
\begin{array}{ll}
h & \text{if } h < n+1;\\ 
h-1 & \text{if } h = n+1.
\end{array}\right.
$$
\end{Proposition}
\begin{proof}
Assume that $h< n+1$. Since, by Theorem \ref{comp} the variety $\mathcal{C}(n,m,h)$ is wonderful with boundary divisors $E_{1},\dots,E_{h-1}$, \cite[Proposition 2.2.1]{Br07} yields an exact sequence
$$ 0 \rightarrow \mathbb{Z}^{h-1} \rightarrow \Pic(\mathcal{C}(n,m,h)) \rightarrow \Pic(\mathcal{C}(n,m,h)^{o}) \rightarrow 0$$
where $\mathbb{Z}^{h-1}$ is the free abelian group generated by the boundary divisors. To conclude it is enough to use Proposition \ref{picg}. The proof in the symmetric case is similar. 
\end{proof}

For $i=1, \dots, h$, we define the divisor $D_i^{\mathcal{C}}$ in $\mathcal{C}(n,m,h)$ as the strict transform of the divisor given by the intersection of $\sec_h(\mathcal{S}^{n,m})$ with 
$$\det \begin{pmatrix}
z_{0,0} & \dots & z_{0,i-1}\\
\vdots & \ddots & \vdots \\
z_{i-1,0} & \dots & z_{i-1,i-1}\\
\end{pmatrix}=0.$$
We will keep the same notation for the corresponding divisors in the intermediate blow-ups $\sec_{h}^{(k)}(\mathcal{S}^{n,m})$.

Similarly, for $i=1, \dots, h$ we define the divisor $D_i^{\mathcal{Q}}$ in $\mathcal{Q}(n,h)$ as the strict transform of the divisor given by the intersection of $\sec_h(\mathcal{V}^n)$ with
$$\det \begin{pmatrix}
z_{0,0} & \dots & z_{0,i-1}\\
\vdots & \ddots & \vdots \\
z_{0,i-1} & \dots & z_{i-1,i-1}\\
\end{pmatrix}=0.$$ 
Again we will keep the same notation for the corresponding divisors in the intermediate blow-ups $\sec_{h}^{(k)}(\mathcal{V}^{n})$.

\begin{Lemma}\label{lem_r_c}
Let $Z$ be an $(n+1) \times (m+1)$ matrix of rank $k < \min\{n+1,m+1\}$ such that the determinant of the top left $k\times k$ minor $Z_k$ of $Z$ vanishes. Then, either the first $k$ rows of $Z$ are linearly dependent or the the first $k$ columns of $Z$ are linearly dependent.
\end{Lemma}
\begin{proof}
Assume that both the first $k$ rows and the first $k$ columns of $Z$ are linearly independent. We will then prove that either $\det(Z_k) \neq 0$ or $\rank(Z) > k$. If $\det(Z_k) \neq 0$ the claim follows. So, assume $\det(Z_k) = 0$. We will write $e_1, \dots, e_{m+1}$ for the canonical basis of $K^{m+1}$ and $\bar{e}_1,\dots \bar{e}_{n+1}$ for the canonical basis of $K^{n+1}$. Since the first $k$ columns of $Z$ are linearly independent, up to a change of coordinates, we may assume that these columns are the vectors $\bar{e}_1,\bar{e}_2, \dots,\bar{e}_{k-1},\bar{e}_{k+1}$. The first $k+1$ rows of the matrix $Z$ are of the following form
$$
\left\lbrace\begin{array}{lll}
Z_{0,-} &= e_1^t+a^0_{k+1}e_{k+1}^t+ \dots +a^0_{m+1}e_{m+1}^t; \\
Z_{1,-} &= e_2^t+ a^1_{k+1}e_{k+1}^t+ \dots +a^1_{m+1}e_{m+1}^t; \\
 & \vdots & \\
Z_{k-2,-} &= e_{k-1}^t+a^{k-2}_{k+1}e_{k+1}^t+ \dots +a^{k-2}_{m+1}e_{m+1}^t; \\
Z_{k-1,-} &= a^{k-1}_{k+1}e_{k+1}^t+ \dots +a^{k-1}_{m+1}e_{m+1}^t; \\
Z_{k,-} &= e_{k}^t+a^{k}_{k+1}e_{k+1}^t+ \dots +a^{k}_{m+1}e_{m+1}^t; \\
\end{array}\right.
$$
for some $a_i^j \in K$. By assumption, the first $k$ rows are linearly independent and so we must have $a^{k-1}_{i} \neq 0$ for at least one $i \in \{k+1, \dots m+1\}$. Hence, the $k+1$ rows $Z_{0,-},\dots,Z_{k,-}$ are linearly independent, and $\rank(Z) \ge k+1$.
\end{proof}

\begin{Corollary}\label{Dksplit}
For $k < \min\{n+1,m+1\}$, the divisor cut out on $\sec_k(\mathcal{S}^{n,m})$ by the top left $k\times k$ minor of the matrix in (\ref{matrix}) has two components $H_1$ and $H_2$, where $H_1$ is cut out by the $k \times k$ minors of the first $k$ rows of $Z$, and $H_2$ is cut out by the $k \times k$ minors of the first $k$ columns of $Z$.
\end{Corollary}
\begin{proof}
The claim follows immediately from Lemma \ref{lem_r_c}.
\end{proof}

\begin{Remark}
In $\sec_k(\mathcal{V}^n)$ the divisor associated to $D_k$ is irreducible. Indeed, in the symmetric case the divisors $H_1,H_2$ in Corollary \ref{Dksplit} coincide.

In order to further clarify this we explicitly work out the case of $3\times 3$ matrices. The hypersurface $D_2 = \{z_{0,0}z_{1,1}-z_{0,1}z_{1,0} = 0\}$ cuts out on $\sec_2(\mathcal{S}^{2,2})\subset\mathbb{P}^8$ a divisor with two irreducible components:
$$H_1 = \{z_{0,1}z_{1,0}-z_{0,0}z_{1,1} = z_{0,2}z_{1,1}-z_{0,1}z_{1,2} = z_{0,2}z_{1,0}- z_{0,0}z_{1,2} = 0\};$$
$$H_2 = \{z_{0,1}z_{1,0}-z_{0,0}z_{1,1} = z_{0,1}z_{2,0}-z_{0,0}z_{2,1}= z_{1,1}z_{2,0}-z_{1,0}z_{2,1} = 0\}.$$
In the symmetric case the divisor $\{z_{0,0}z_{1,1}-z_{0,1}^2 = 0\}$ cuts out on $\sec_2(\mathcal{V}^2)\subset\mathbb{P}^5$ the irreducible divisor
$$\{z_{0,2}z_{1,1}-z_{0,1}z_{1,2} = z_{0,1}z_{0,2}-z_{0,0}z_{1,2} = z_{0,1}^2-z_{0,0}z_{1,1} = 0\}$$
with multiplicity two.    
\end{Remark}

\begin{Notation}
We will denote by $H_1^{\mathcal{C}},H_2^{\mathcal{C}}$ the strict transforms of $H_1,H_2$ in $\mathcal{C}(n,m,h)$.
\end{Notation}

\begin{Proposition}\label{col_bou}
The set of colors of $\mathcal{C}(n,m,h)$ is given by
$$
\begin{array}{ll}
\{D_1^{\mathcal{C}}, \dots, D_n^{\mathcal{C}}\} & \text{if } h = n+1 = m+1;\\ 
\{H_1^{\mathcal{C}},H_2^{\mathcal{C}}, D_1^{\mathcal{C}}, \dots, D_{h-1}^{\mathcal{C}}\} & \text{if } h<n+1;\\
\{D_1^{\mathcal{C}}, \dots, D_{n+1}^{\mathcal{C}}\} & \text{if } h = n+1 < m+1;
\end{array}
$$
while for $\mathcal{Q}(n,h)$ the set of colors is given by
$$
\begin{array}{ll}
\{D_1^{\mathcal{Q}}, \dots, D_{h}^{\mathcal{Q}}\} & \text{if } h < n+1;\\ 
\{D_1^{\mathcal{Q}}, \dots, D_{n}^{\mathcal{Q}}\}\ & \text{if } h = n+1.
\end{array}
$$

\end{Proposition}
\begin{proof}
The claim for $\mathcal{C}(n,m)$ and $\mathcal{Q}(n)$ follows from \cite[Proposition 3.6]{Ma18a}. In particular, the divisors listed in the statement are stabilized by the action of the Borel subgroups in (\ref{actcol}) and (\ref{actquad}) respectively. Moreover, $\sec_h(\mathcal{S}^{n,m})$ and $\sec_h(\mathcal{V}^n)$ are stabilized respectively by the action (\ref{actcol}) and (\ref{actquad}). Then, $D_1^{\mathcal{C}},\dots, D_h^{\mathcal{C}}$ are stabilized by the restriction of the action (\ref{actcol}), and similarly the strict transform in $\mathcal{Q}(n,h)$ of $D_1^{\mathcal{Q}},\dots, D_h^{\mathcal{Q}}$ are stabilized by the restriction of the action (\ref{actquad}).

The groups acting are connected, so any reducible divisor which is stabilized must be stabilized component wise. In particular, since by Corollary \ref{Dksplit} in $\mathcal{C}(n,m,h)$ for $h < n+1$ we have $D_h^{\mathcal{C}} = H_1^{\mathcal{C}}\cup H_2^{\mathcal{C}}$ and since $D_h^{\mathcal{C}}$ is stabilized, we have that both $H_1^{\mathcal{C}}$ and $H_2^{\mathcal{C}}$ are stabilized. 

As noticed in \cite[Remark 4.5.5.3]{ADHL15}, if $(X,\mathscr{G},\mathscr{B},x_0)$ is a spherical wonderful variety with colors $D_1,\dots,D_s$ the big cell $X\setminus (D_1\cup\dots \cup D_s)$ is an affine space. Therefore, it admits only constant invertible global functions and $\Pic(X) = \mathbb{Z}[D_1,\dots,D_s]$.

Now, for $h < n+1$ in $\mathcal{C}(n,m,h)$ we have $h+1$ colors and since by Proposition \ref{picrank} the Picard rank of $\mathcal{C}(n,m,h)$ is $h+1$, these divisors $D_1, \dots, D_{h-1}, H_1^{\mathcal{C}},H_2^{\mathcal{C}}$ must be all the colors. Similarly, for $h=n+1<m+1$ we found the divisors $D_1^{\mathcal{C}}, \dots, D_{n+1}^{\mathcal{C}}$, and since in this case $\rho(\mathcal{C}(n,m,h))=h$, they are all the colors. Note that when $h=n+1=m+1$, the divisor $D_{n+1}^{\mathcal{C}}$ is not a color, since it is stabilized by the whole group. In this case $\rho(\mathcal{C}(n,m,h))=h-1$ and then $D_1^{\mathcal{C}}, \dots, D_{n}^{\mathcal{C}}$ are the colors. With a similar argument we can compute the colors of $\mathcal{Q}(n,h)$.
\end{proof}

\begin{Proposition}\label{effnef}
For the effective and the nef cone of $\mathcal{C}(n,m,h)$ we have 
$$\Eff(\mathcal{C}(n,m,h))= \left\lbrace
\begin{array}{ll}
\langle E_1^{\mathcal{C}}, \dots, E_{h-1}^{\mathcal{C}} \rangle  &\text{ if } h = n+1 = m+1;\\ 
\langle E_1^{\mathcal{C}}, \dots, E_{h-1}^{\mathcal{C}}, H_1^{\mathcal{C}},H_2^{\mathcal{C}} \rangle   &\text{if } h<n+1;\\
\langle E_1^{\mathcal{C}}, \dots, E_{h-1}^{\mathcal{C}}, D_{n+1}^{\mathcal{C}} \rangle  &\text{if } h = n+1 < m+1;
\end{array}\right.
$$
$$
\Nef(\mathcal{C}(n,m,h))= \left\lbrace
\begin{array}{ll}
\langle D_1^{\mathcal{C}}, \dots, D_n^{\mathcal{C}} \rangle  &\text{ if } h = n+1 = m+1;\\ 
\langle D_1^{\mathcal{C}}, \dots, D_{h-1}^{\mathcal{C}}, H_1^{\mathcal{C}},H_2^{\mathcal{C}} \rangle   &\text{if } h<n+1;\\
\langle D_1^{\mathcal{C}}, \dots, D_{n+1}^{\mathcal{C}} \rangle  &\text{if } h = n+1 < m+1;
\end{array}\right.
$$
and for the effective and the nef cone of $\mathcal{Q}(n,h)$ we have 
$$
\Eff(\mathcal{Q}(n,h))= \left\lbrace
\begin{array}{ll}
\langle E_1^{\mathcal{Q}}, \dots, E_{h-1}^{\mathcal{Q}},D_{h}^{\mathcal{Q}} \rangle & \text{ if } h < n+1;\\ 
\langle E_1^{\mathcal{Q}}, \dots, E_{h-1}^{\mathcal{Q}} \rangle & \text{if } h = n+1;
\end{array}\right.$$
$$\Nef(\mathcal{Q}(n,h))= \left\lbrace
\begin{array}{ll}
\langle D_1^{\mathcal{Q}}, \dots, D_h^{\mathcal{Q}} \rangle & \text{ if } h < n+1;\\ 
\langle D_1^{\mathcal{Q}}, \dots, D_{n}^{\mathcal{Q}} \rangle  & \text{if } h=n+1.\\
\end{array}\right.$$
\end{Proposition}

\begin{proof}
The statement for $\mathcal{C}(n,m)$ and $\mathcal{Q}(n)$ follows from \cite[Theorem 3.13]{Ma18a}. We consider now the case $h < n+1$.

Consider $\mathcal{C}(n,m,h)$. By \cite[Proposition 4.5.4.4]{ADHL15}, Theorem \ref{comp} and Proposition \ref{effnef} the effective cone of $\mathcal{C}(n,m,h)$ is generated by $E_1^{\mathcal{C}}, \dots, E^{\mathcal{C}}_{h-1}, D_1^{\mathcal{C}}, \dots, D_{h-1}^{\mathcal{C}}, H_1^{\mathcal{C}},H_2^{\mathcal{C}}$. By \cite[Section 5]{Ma18a} the divisor $D_i^{\mathcal{C}}$ induces a birational morphism that contracts the exceptional divisor $E^{\mathcal{C}}_i$. Therefore $D_i^{\mathcal{C}}$ lies in the interior of the effective cone for any $i=1, \dots, h-1$. In particular, since by Proposition \ref{picrank} $\rho(\mathcal{C}(n,m,h))=h+1$, we conclude that the extremal rays of the effective cone are $E_1^{\mathcal{C}}, \dots, E_{h-1}^{\mathcal{C}}, H_1^{\mathcal{C}},H_2^{\mathcal{C}}$. 

Furthermore, by \cite[Section 2.6]{Br89} the nef cone is generated by $D_1^{\mathcal{C}}, \dots, D_{h-1}^{\mathcal{C}}, H_1^{\mathcal{C}},H_2^{\mathcal{C}}$. A similar argument gives the generators for the effective and nef cone of $\mathcal{Q}(n,h)$.
\end{proof}

\subsection{Birational geometry of $\mathcal{C}(n,m,h)$ and $\mathcal{Q}(n,h)$}
Let $X$ be a normal $\mathbb{Q}$-factorial variety. We say that a birational map  $f: X \dasharrow X'$ to a normal projective variety $X'$  is a \emph{birational contraction} if its
inverse does not contract any divisor. 
We say that it is a \emph{small $\mathbb{Q}$-factorial modification} 
if $X'$ is $\mathbb{Q}$-factorial and $f$ is an isomorphism in codimension one.
If $f: X \dasharrow X'$ is a small $\mathbb{Q}$-factorial modification then 
the natural pullback map $f^*:N^1(X')\to N^1(X)$ sends $\Mov(X')$ and $\Eff(X')$
isomorphically onto $\Mov(X)$ and $\Eff(X)$ respectively. In particular, we have $f^*(\Nef(X'))\subset \overline{\Mov(X)}$.

Now, assume that the divisor class group $\Cl(X)$ is free and finitely generated, and fix a subgroup $G$ of the group of Weil divisors on $X$ such that the canonical map $G\rightarrow\Cl(X)$, mapping a divisor $D\in G$ to its class $[D]$, is an isomorphism. The \textit{Cox ring} of $X$ is defined as
$$\Cox(X) = \bigoplus_{[D]\in \Cl(X)}H^0(X,\mathcal{O}_X(D))$$
where $D\in G$ represents $[D]\in\Cl(X)$, and the multiplication in $\Cox(X)$ is defined by the standard multiplication of homogeneous sections in the field of rational functions on $X$. 

\begin{Definition}\label{def:MDS} 
A normal projective $\mathbb{Q}$-factorial variety $X$ is called a \emph{Mori dream space}
if the following conditions hold:
\begin{enumerate}
\item[-] $\Pic{(X)}$ is finitely generated, or equivalently $h^1(X,\mathcal{O}_X)=0$,
\item[-] $\Nef{(X)}$ is generated by the classes of finitely many semi-ample divisors,
\item[-] there is a finite collection of small $\mathbb{Q}$-factorial modifications
 $f_i: X \dasharrow X_i$, such that each $X_i$ satisfies the second condition above, and $
 \Mov{(X)} \ = \ \bigcup_i \  f_i^*(\Nef{(X_i)})$.
\end{enumerate}
\end{Definition}

The collection of all faces of all cones $f_i^*(\Nef{(X_i)})$ above forms a fan which is supported on $\Mov(X)$.
If two maximal cones of this fan, say $f_i^*(\Nef{(X_i)})$ and $f_j^*(\Nef{(X_j)})$, meet along a facet,
then there exist a normal projective variety $Y$, a small modification $\varphi:X_i\dasharrow X_j$, and $h_i:X_i\rightarrow Y$, $h_j:X_j\rightarrow Y$ small birational morphisms of relative Picard number one such that $h_j\circ\varphi = h_i$. The fan structure on $\Mov(X)$ can be extended to a fan supported on $\Eff(X)$ as follows. 

\begin{Definition}\label{MCD}
Let $X$ be a Mori dream space.
We describe a fan structure on the effective cone $\Eff(X)$, called the \emph{Mori chamber decomposition}. There are finitely many birational contractions from $X$ to Mori dream spaces, denoted by $g_i:X\dasharrow Y_i$.
The set $\Exc(g_i)$ of exceptional prime divisors of $g_i$ has cardinality $\rho(X/Y_i)=\rho(X)-\rho(Y_i)$.
The maximal cones $\mathcal{C}$ of the Mori chamber decomposition of $\Eff(X)$ are of the form $\mathcal{C}_i \ = \left\langle g_i^*\big(\Nef(Y_i)\big) , \Exc(g_i) \right\rangle$. We call $\mathcal{C}_i$ or its interior $\mathcal{C}_i^{^\circ}$ a \emph{maximal chamber} of $\Eff(X)$. We refer to \cite[Proposition 1.11]{HK00} and \cite[Section 2.2]{Ok16} for details.
\end{Definition}

\begin{Remark}\label{sphMDS}
By the work of M. Brion \cite{Br93} we have that $\mathbb{Q}$-factorial spherical varieties are Mori dream spaces. An alternative proof of this result can be found in \cite[Section 4]{Pe14}. In particular, by Theorem \ref{comp} $\mathcal{C}(n,m,h)$ and $\mathcal{Q}(n,h)$ are Mori dream spaces.
\end{Remark}

\begin{Remark}\label{toric}
Recall that by \cite[Proposition 2.11]{HK00} given a Mori Dream Space $X$ there is an embedding $i:X\rightarrow \mathcal{T}_X$ into a simplicial projective toric variety $\mathcal{T}_X$ such that $i^{*}:\Pic(\mathcal{T}_X)\rightarrow \Pic(X)$ is an isomorphism inducing an isomorphism $\Eff(\mathcal{T}_X)\rightarrow \Eff(X)$. Furthermore, the Mori chamber decomposition of $\Eff(\mathcal{T}_X)$ is a refinement of the Mori chamber decomposition of $\Eff(X)$. Indeed, if $\Cox(X) \cong \frac{K[T_1,\dots,T_s]}{I}$ where the $T_i$ are homogeneous generators with non-trivial effective $\Pic(X)$-degrees then $\Cox(\mathcal{T}_X)\cong K[T_1,\dots,T_s]$.

Since the variety $\mathcal{T}_{X}$ is toric, the Mori chamber decomposition of $\Eff(\mathcal{T}_{X})$ can be computed by means of the Gelfand–Kapranov–Zelevinsky, GKZ for short, decomposition \cite[Section 2.2.2]{ADHL15}. Let us consider the family $\mathcal{W}$ of vectors in $\Pic(\mathcal{T}_{X})$ given by the generators of $\Cox(\mathcal{T}_{X})$, and let $\Omega(\mathcal{W})$ be the set of all convex polyhedral cones generated by some of the vectors in $\mathcal{W}$. By \cite[Construction 2.2.2.1]{ADHL15} the GKZ chambers of $\Eff(\mathcal{T}_{X})$ are given by the intersections of all the cones in $\Omega(\mathcal{W})$ containing a fixed divisor in $\Eff(\mathcal{T}_{X})$.
\end{Remark}

\begin{Remark}\label{gen_cox}
Let $(X,\mathscr{G},\mathscr{B},x_0)$ be a projective spherical variety. Consider a divisor $D$ on $X$, and let $f_D$ be the, unique up to constants, section of $\mathcal{O}_X(D)$ associated to $D$. We will denote by $\lin_K(\mathscr{G}\cdot D)\subseteq\Cox(X)$ the finite-dimensional vector subspace of $\Cox(X)$ spanned by the orbit of $f_D$ under the action of $\mathscr{G}$ that is the smallest linear subspace of $\Cox(X)$ containing the $\mathscr{G}$-orbit of $f_D$.

By \cite[Theorem 4.5.4.6]{ADHL15} if $\mathscr{G}$ is a semi-simple and simply connected algebraic group and $(X,\mathscr{G},\mathscr{B},x_0)$ is a spherical variety with boundary divisors $E_1,\dots,E_r$ and colors $D_1,\dots,D_s$ then $\Cox(X)$ is generated as a $K$-algebra by the canonical sections of the $E_i$ and the finite dimensional vector subspaces $\lin_{K}(\mathscr{G}\cdot D_i)\subseteq \Cox(X)$ for $1\leq i\leq s$.
\end{Remark}

Next, we study the birational geometry of $\mathcal{C}(n,m,h)$ and $\mathcal{Q}(n,h)$ when the Picard rank is small. We begin with $\mathcal{Q}(n,h)$. The varieties $\mathcal{Q}(1,2)$ and $\mathcal{Q}(2,3)$ are covered by \cite[Section 6]{Ma18a}. So, the first case to consider is that of $\mathcal{Q}(n,3)$ for $n \ge 3$.

\begin{Lemma}\label{div_dec}
For the variety $\mathcal{Q}(n,3)$ we have that $D_1^{\mathcal{Q}} \sim H$, $D_2^{\mathcal{Q}} \sim 2H-E_1^{\mathcal{Q}}$, $D_3^{\mathcal{Q}} \sim 3H-2E_1^{\mathcal{Q}}-E_2^{\mathcal{Q}}$.
\end{Lemma}
\begin{proof}
Consider the strict transform $L\subset\mathcal{Q}(n,3)$ of the line $L_{\mu,\lambda}=\{\mu x_0^2+\lambda(x_1^2+x_2^2)=0\}$. This line intersects $\mathcal{V}^n$ at a point $p$, $\sec_2(\mathcal{V}^n)\setminus \mathcal{V}^n$ at a point $q$, and it is not contained neither in the tangent cone of $\sec_2(\mathcal{V}^n)$ at $p$ nor in the tangent space of $\overline{H}_2 = \{z_{0,0}z_{1,1}-z_{0,1}^2=0\}$ at $p$. 

First, consider the blow-up $\sec_3^{(1)}(\mathcal{V}^n)$ of $\sec_3(\mathcal{V}^n)$ along $\mathcal{V}^n$ and keep the same notation for the push-forward to $\sec_3^{(1)}(\mathcal{V}^n)$ of $L$, $D_1^{\mathcal{Q}}, D_2^{\mathcal{Q}}, D_3^{\mathcal{Q}}$. Recall that $\sec_3^{(1)}(\mathcal{V}^n)$ is singular along the strict transform of $\sec_2(\mathcal{V}^n)$. However, $D_2^{\mathcal{Q}}, D_3^{\mathcal{Q}}$ are Cartier on $\sec_3^{(1)}(\mathcal{V}^n)$ since they are restrictions to $\sec_3^{(1)}(\mathcal{V}^n)$ of divisors in the blow-up of $\mathbb{P}^{N_{+}}$ along $\mathcal{V}^n$.

Write $D_2^{\mathcal{Q}} = 2H-aE_1$. Note that $\overline{H}_2$ intersect $L_{\mu,\lambda}$ at $p$. Since $L_{\mu,\lambda}$ is not contained in the tangent space of $\overline{H}_2$ at $p$ in $\sec_3^{(1)}(\mathcal{V}^n)$ the strict transforms $L$ and $D_2^{\mathcal{Q}}$ intersect just in one point. Then $1 = D_2^{\mathcal{Q}}\cdot L = 2-a$ yields $a = 1$. 

Similarly, $L_{\mu,\lambda}$ intersects the cubic hypersurface $\overline{H}_3$ given be the top left $3\times 3$ minor of (\ref{matrix}) at $p$ with multiplicity two and at $q$. Moreover, $L_{\mu,\lambda}$ is not contained in the tangent cone of $\overline{H}_3$ at $p$ and hence in $\sec_3^{(1)}(\mathcal{V}^n)$ the strict transforms $L$ and $D_3^{\mathcal{Q}}$ intersect just in one point. Then $1 = D_3^{\mathcal{Q}}\cdot L$. Writing $D_3^{\mathcal{Q}} \sim 3H-bE_1$ we get $1 = D_3^{\mathcal{Q}}\cdot L = 3-b$ and hence $b = 2$.

Now, we consider $\mathcal{Q}(n,3)$. Since $D_2^{\mathcal{Q}}$ does not contain the strict transform of $\sec_2(\mathcal{V}^n)$ its expression remains unvaried after the last blow-up. On the other hand, $E_2$ must appear in the expression of $D_3^{\mathcal{Q}}$. Let us write $D_3^{\mathcal{Q}} \sim 3H-2E_1-cE_2$ and keep denoting by $L$ its strict transform in $\mathcal{Q}(n,3)$. Note that $L_{\mu,\lambda}$ is not contained in the tangent space of $\overline{H}_3$ at $q$. So $0 = D_3^{\mathcal{Q}}\cdot L = 3-2-c$ and hence $c = 1$.
\end{proof}

\begin{Proposition}\label{mcdq3}
For $n \ge 3$, the Mori chamber decomposition of $\Eff(\mathcal{Q}(n,3))$ has five chambers as displayed in the following $2$-dimension section of $\Eff(\mathcal{Q}(n,3))$
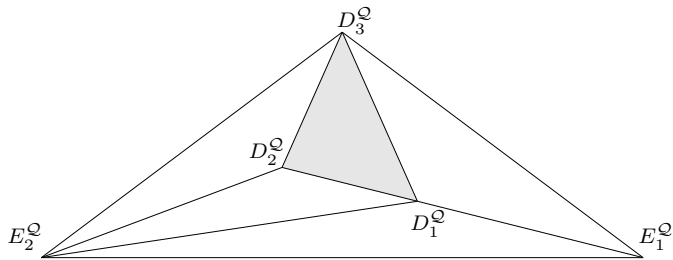
$$
\definecolor{uuuuuu}{rgb}{0.26666666666666666,0.26666666666666666,0.26666666666666666}
\begin{tikzpicture}[xscale=4,yscale=3][line cap=round,line join=round,>=triangle 45,x=1cm,y=1cm]
\clip(-1.7300693940731935,-0.10) rectangle (2.1061218447714393,1.15);
\fill[line width=0.4pt,fill=black,fill opacity=0.10000000149011612] (0,1) -- (-0.2,0.4) -- (0.25,0.25) -- cycle;
\draw [line width=0.1pt] (-1,0)-- (0,1);
\draw [line width=0.1pt] (0,1)-- (1,0);
\draw [line width=0.1pt] (1,0)-- (-1,0);
\draw [line width=0.1pt] (-0.2,0.4)-- (0.25,0.25);
\draw [line width=0.1pt] (-0.2,0.4)-- (0,1);
\draw [line width=0.1pt] (0,1)-- (0.25,0.25);
\draw [line width=0.1pt] (0.25,0.25)-- (1,0);
\draw [line width=0.1pt] (-0.2,0.4)-- (-1,0);
\draw [line width=0.1pt] (-1,0)-- (0.25,0.25);
\begin{scriptsize}
\draw[color=black] (1.0432914585343922,0.09182320667404484) node {$E_1^{\mathcal{Q}}$};
\draw[color=black] (-1.0539194341399391,0.09182320667404484) node {$E_2^{\mathcal{Q}}$};
\draw[color=black] (0.042436715582343396,1.05) node {$D_3^{\mathcal{Q}}$};
\draw[color=black] (0.28692031691414155,0.15) node {$D_1^{\mathcal{Q}}$};
\draw[color=black] (-0.25,0.47) node {$D_2^{\mathcal{Q}}$};
\end{scriptsize}
\end{tikzpicture}$$
where $\Mov(\mathcal{Q}(n,3))$ coincides with $\Nef(\mathcal{Q}(n,3))$ and is generated by $D_1^{\mathcal{Q}},D_2^{\mathcal{Q}},D_3^{\mathcal{Q}}$.
\end{Proposition}
\begin{proof}
By Theorem \ref{comp}, Proposition \ref{col_bou}, Remarks \ref{toric}, \ref{gen_cox}, and Lemma \ref{div_dec} the Mori chamber decomposition of $\Eff(\mathcal{Q}(n,3))$ is a possibly trivial coarsening of the decomposition in the
statement. 

Since by Proposition \ref{effnef} $D_1^{\mathcal{Q}},D_2^{\mathcal{Q}},D_3^{\mathcal{Q}}$  are the generators of the nef cone of $\mathcal{Q}(n,3)$, these rays can not be removed. Furthermore, since Mori chamber are convex the walls between $E_{2}^{\mathcal{Q}},D_2^{\mathcal{Q}}$ and $E_{1}^{\mathcal{Q}},D_1^{\mathcal{Q}}$ can not be removed. Finally, to see that the wall between $E_{2}^{\mathcal{Q}},D_1^{\mathcal{Q}}$ can not be removed it is enough to observe that the stable base locus of a divisor in the chamber delimited by $E_{2}^{\mathcal{Q}},D_2^{\mathcal{Q}},D_1^{\mathcal{Q}}$ is $E_2^{\mathcal{Q}}$, while the stable base locus of a divisor in the chamber delimited by $E_{2}^{\mathcal{Q}},D_1^{\mathcal{Q}},E_1^{\mathcal{Q}}$ is $E_1^{\mathcal{Q}}\cup E_2^{\mathcal{Q}}$.
\end{proof}

We will study the decomposition of the effective cone of $\mathcal{C}(n,m,2)$. For $n=m=1$ we have $\mathcal{C}(1,1,2)\cong\mathbb{P}^
3$. Hence, the first interesting cases occur for $n=1$ and $m > 1$. The case $\mathcal{C}(1,m,2)$ is in \cite[Page 1606]{Ma18a}.

\begin{Proposition}\label{mcd_C}
For $n >1$ and $m >1$ the Mori chamber decomposition of $\Eff(\mathcal{C}(n,m,2))$ has three chambers as displayed in the following $2$-dimensional section of $\Eff(\mathcal{C}(n,m,2))$
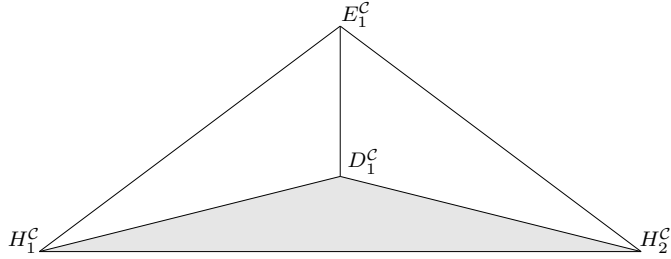
$$\definecolor{uuuuuu}{rgb}{0.26666666666666666,0.26666666666666666,0.26666666666666666}
\begin{tikzpicture}[xscale=4,yscale=3][line cap=round,line join=round,>=triangle 45,x=1cm,y=1cm]
\clip(-1.7300693940731935,-0.10) rectangle (2.1061218447714393,1.2);
\fill[line width=0.4pt,fill=black,fill opacity=0.10000000149011612] (-1,0) -- (0,0.3333333333333333) -- (1,0) -- cycle;
\draw [line width=0.1pt] (-1,0)-- (0,1);
\draw [line width=0.1pt] (0,1)-- (1,0);
\draw [line width=0.1pt] (1,0)-- (-1,0);
\draw [line width=0.1pt] (-1,0)-- (0,0.3333333333333333);
\draw [line width=0.1pt] (0,0.3333333333333333)-- (1,0);
\draw [line width=0.1pt] (0,0.3333333333333333)-- (0,1);
\begin{scriptsize}
\draw[color=black] (1.0495917076549072,0.05) node {$H_2^{\mathcal{C}}$};
\draw[color=black] (-1.0499914661205724,0.05) node {$H_1^{\mathcal{C}}$};
\draw[color=black] (0.053261729358672814,1.05) node {$E_1^{\mathcal{C}}$};
\draw[color=black] (0.07756246053662977,0.4) node {$D_1^{\mathcal{C}}$};
\end{scriptsize}
\end{tikzpicture}$$
where $\Mov(\mathcal{C}(n,m,2))$ coincides with $\Nef(\mathcal{C}(n,m,2))$ and is generated by $H_1^{\mathcal{C}},H_2^{\mathcal{C}},D_1^{\mathcal{C}}$.
\end{Proposition}
\begin{proof}
It is enough to argue as in the proof of Proposition \ref{mcdq3}, and to observe that since Mori chambers are convex in the case $n>1,m>1$ the wall between $E_1^{\mathcal{C}},D_1^{\mathcal{C}}$ can not be removed.
\end{proof}

In the following we consider the spherical variety $\sec_4^{(2)}(\mathcal{V}^n)$ obtained by blowing-up $\sec_4(\mathcal{V}^n)$ along $\mathcal{V}^n$ and then along the strict transform of $\sec_2(\mathcal{V}^n)$. We will keep denoting by $D_i^{\mathcal{Q}}, E_j^{\mathcal{Q}}$ the push-forward of the corresponding divisors via the blow-down $\mathcal{Q}(n,4)\rightarrow\sec_4^{(2)}(\mathcal{V}^n)$. 

\begin{Proposition}\label{MCD_G}
The Mori chamber decomposition of $\Eff(\sec_4^{(2)}(\mathcal{V}^n))$ has nine chambers as displayed in the following $2$-dimensional section of $\Eff(\sec_4^{(2)}(\mathcal{V}^n))$ 
$$
\begin{tikzpicture}[xscale=0.4,yscale=0.7][line cap=round,line join=round,>=triangle 45,x=1cm,y=1cm]\clip(-14.9,-0.21) rectangle (14.5,6.55);\fill[line width=0pt,fill=black,fill opacity=0.15] (-5.000432432432432,2.4614054054054053) -- (5.000432432432432,2.4614054054054053) -- (0,4) -- cycle;\fill[line width=0pt,color=wwwwww,fill=white,fill opacity=0.15] (-5.000432432432432,2.4614054054054053) -- (5.000432432432432,2.4614054054054053) -- (0,1.7776389756402244) -- cycle;\draw [line width=0.1pt] (-13,0)-- (13,0);\draw [line width=0.1pt] (13,0)-- (0,6);\draw [line width=0.1pt] (0,6)-- (-13,0);\draw [line width=0.1pt] (0,4)-- (-13,0);\draw [line width=0.1pt] (0,4)-- (13,0);\draw [line width=0.1pt] (-5.000432432432432,2.4614054054054053)-- (13,0);\draw [line width=0.1pt] (5.000432432432432,2.4614054054054053)-- (-13,0);\draw [line width=0.1pt] (-5.000432432432432,2.4614054054054053)-- (5.000432432432432,2.4614054054054053);\draw [line width=0.1pt] (-5.000432432432432,2.4614054054054053)-- (0,6);\draw [line width=0.1pt] (0,6)-- (5.000432432432432,2.4614054054054053);\draw [line width=0.1pt] (0,4)-- (0,6);\begin{scriptsize}\draw [fill=black] (-13,0) circle (0pt);\draw[color=black] (-13.7,0.3) node {$E_1^{\mathcal{Q}}$};\draw [fill=black] (13,0) circle (0pt);\draw[color=black] (13.6,0.3) node {$D_4^{\mathcal{Q}}$};\draw [fill=black] (0,6) circle (0pt);\draw[color=black] (0.18536585365853658,6.21) node {$E_2^{\mathcal{Q}}$};\draw [fill=black] (0,4) circle (0pt);\draw[color=black] (0.95,4.2) node {$D_2^{\mathcal{Q}}$};\draw [fill=black] (-5.000432432432432,2.4614054054054053) circle (0pt);\draw[color=black] (-5.95,2.7) node {$D_1^{\mathcal{Q}}$};\draw [fill=black] (5.000432432432432,2.4614054054054053) circle (0pt);\draw[color=black] (5.75,2.7) node {$D_3^{\mathcal{Q}}$};\draw [fill=uuuuuu] (0,1.7776389756402244) circle (0pt);\draw[color=uuuuuu] (0.18536585365853658,1.4) node {$P$};\end{scriptsize}\end{tikzpicture}
$$
where $\Nef(\sec_4^{(2)}(\mathcal{V}^n))$ is generated by $D_1^{\mathcal{Q}},D_2^{\mathcal{Q}},D_3^{\mathcal{Q}}$, and the movable cone $\Mov(\sec_4^{(2)}(\mathcal{V}^n))$ is generated by $D_1^{\mathcal{Q}},D_2^{\mathcal{Q}},D_3^{\mathcal{Q}},P$ with $P \sim 6 D_1^{\mathcal{Q}}-3E_1^{\mathcal{Q}}-2E_2^{\mathcal{Q}}$.
\end{Proposition}
\begin{proof}
Note that the $SL(n+1)$-actions on $\sec_4^{(2)}(\mathcal{V}^n)$ and $\mathcal{Q}(n,4)$ are equivariant with respect to the blow-down morphism $\mathcal{Q}(n,4)\rightarrow\sec_4^{(2)}(\mathcal{V}^n)$. Hence, by Proposition \ref{col_bou} the colors of $\sec_4^{(2)}(\mathcal{V}^n)$ are $D_1^{\mathcal{Q}},D_2^{\mathcal{Q}},D_3^{\mathcal{Q}},D_4^{\mathcal{Q}}$, and its boundary divisors are $E_1^{\mathcal{Q}},E_2^{\mathcal{Q}}$. Arguing as in the proof of Lemma \ref{div_dec} we have that $D_4^{\mathcal{Q}}\sim 4H-3E_1^{\mathcal{Q}}-2E_2^{\mathcal{Q}}$. Note that $D_4^{\mathcal{Q}}$ is also a boundary divisor when $n = 3$. Now, the claim on the movable cone follows from Remark \ref{gen_cox} and \cite[Proposition 3.3.2.3]{ADHL15}. Finally, to conclude it is enough to argue as in the proof of Proposition \ref{mcdq3}.
\end{proof}

We conclude this section by computing the automorphism groups of the varieties $\sec_h^{(k)}(\mathcal{S}^{n,m})$ and $\sec_h^{(k)}(\mathcal{V}^{n})$.

\begin{Proposition}\label{aut_sec}
For all $h\leq n$ we have 
$$
\Aut(\sec_h(\mathcal{S}^{n,m})) \cong
\left\lbrace\begin{array}{ll}
PGL(n+1)\times PGL(m+1) & \text{if n}< \textit{m};\\ 
S_2 \ltimes (PGL(n+1)\times PGL(n+1)) & \text{if n} = \textit{m};
\end{array}\right.
$$
and $\Aut(\sec_h(\mathcal{V}^{n}))\cong  PGL(n+1)$.
\end{Proposition}
\begin{proof}
Let $\phi$ be an automorphism of $\sec_h(\mathcal{S}^{n,m})$. By the stratification of the singular locus of $\sec_h(\mathcal{S}^{n,m})$ described in Proposition \ref{tcones} $\phi$ must stabilize $\sec_k(\mathcal{S}^{n,m})$ for all $k\leq h$. In particular, $\phi$ induces an automorphism $\phi_{|\mathcal{S}^{n,m}}\in \Aut(\mathcal{S}^{n,m})$, and by \cite[Lemma 7.4]{Ma18a} we have that $\Aut(\mathcal{S}^{n,m})\cong PGL(n+1)\times PGL(m+1)$ if $n < m$, and $\Aut(\mathcal{S}^{n,n})\cong S_2 \ltimes (PGL(n+1)\times PGL(n+1))$. 

Note that in the case $n = m$ also the involution in $S_2$ switching the two factors comes from an automorphism of the ambient projective space $\mathbb{P}^N$ and so it induces an automorphism of $\sec_h(\mathcal{S}^{n,m})$. Let us proceed by induction on $h$. So 
$\Aut(\sec_{h-1}(\mathcal{S}^{n,m}))\cong \Aut(\mathcal{S}^{n,m})$, and we have a surjective morphism of groups 
$$
\begin{array}{cccc}
\chi: &\Aut(\sec_h(\mathcal{S}^{n,m}))& \longrightarrow & \Aut(\sec_{h-1}(\mathcal{S}^{n,m}))\\
      & \phi & \longmapsto & \phi_{|\sec_{h-1}(\mathcal{S}^{n,m})}.
\end{array}
$$
Recall that $\sec_h(\mathcal{S}^{n,m}) = \Jn(\sec_{h-1}(\mathcal{S}^{n,m}), \mathcal{S}^{n,m})$. Assume that $\phi_{|\sec_{h-1}(\mathcal{S}^{n,m})} = Id_{\sec_{h-1}(\mathcal{S}^{n,m})}$. Then $\phi_{|\sec_{h-1}(\mathcal{S}^{n,m})}$ fixes $\sec_{h-1}(\mathcal{S}^{n,m})$ and hence $\mathcal{S}^{n,m}$. Let $p\in \sec_h(\mathcal{S}^{n,m})$ be a general point. By Remark \ref{sec_ver} the actual dimension of $\Jn(\sec_{h-1}(\mathcal{S}^{n,m}), \mathcal{S}^{n,m})$ is smaller than the expected one. So there are infinitely many lines intersecting $\mathcal{S}^{n,m}$ and $\sec_{h-1}(\mathcal{S}^{n,m})$ through $p$. Any two of these lines are stabilized by $\phi$ and intersect at $p$, so $\phi(p) = p$. Hence $\phi = Id_{\sec_h(\mathcal{S}^{n,m})}$ and $\chi$ is an isomorphism. The same proof, with the obvious variations, works in the symmetric case as well.
\end{proof}

\begin{thm}\label{autgr}
For all $h\leq n$ and $k = 1,\dots, h-1$ we have
$$
\Aut(\sec_h^{(k)}(\mathcal{S}^{n,m})) \cong
\left\lbrace\begin{array}{ll}
PGL(n+1)\times PGL(m+1) & \text{if n}< \textit{m};\\ 
S_2 \ltimes (PGL(n+1)\times PGL(n+1)) & \text{if n} = \textit{m};
\end{array}\right.
$$
$$\Aut(\sec_h^{(k)}(\mathcal{V}^{n}))\cong  PGL(n+1);$$
and for $h = n+1$ we have
$$
\Aut(\mathcal{C}(n,m,n+1)) \cong
\left\lbrace\begin{array}{ll}
PGL(n+1)\times PGL(m+1) & \text{if n}< \textit{m};\\ 
(S_2 \ltimes (PGL(n+1)\times PGL(n+1)))\rtimes S_2 & \text{if n} = \textit{m}\geq 2;
\end{array}\right.
$$
$$
\Aut(\mathcal{Q}(n,n+1))\cong PGL(n+1)\rtimes S_2;
$$
$\Aut(\mathcal{C}(1,1,2))\cong PGL(4)$, and $\Aut(\mathcal{Q}(1,2))\cong PGL(3)$.
\end{thm}
\begin{proof}
When $h = n+1$ the statement follows from \cite[Theorem 7.5]{Ma18a}. Hence we consider the case $h\leq n$. We will prove the claim for $\sec_h^{(k)}(\mathcal{S}^{n,m})$. The argument in the symmetric case is completely analogous. 

First, take $k = h-1$. Hence $\sec_h^{(h-1)}(\mathcal{S}^{n,m})\cong\mathcal{C}(n,m,h)$. An automorphism $\phi\in \Aut(\mathcal{C}(n,m,h))$ acts on the extremal rays of $\Eff(\mathcal{C}(n,m,h))$ as a permutation. If it acts non trivially then it must act non trivially also on the generators of $\Nef(\mathcal{C}(n,m,h))$ in Proposition \ref{effnef}. However, this is not possible since for instance these nef divisors have spaces of global sections of different dimensions. Hence, $\phi$ stabilizes all the exceptional divisors in Definition \ref{def1}, and therefore it induces an automorphism $\widetilde{\phi}\in \Aut(\sec_h(\mathcal{S}^{n,m}))$. The morphism of groups 
$$
\begin{array}{cccc}
\widetilde{\chi}: &\Aut(\mathcal{C}(n,m,h))& \longrightarrow & \Aut(\sec_{h}(\mathcal{S}^{n,m}))\\
      & \phi & \longmapsto & \widetilde{\phi}
\end{array}
$$
is clearly an isomorphism, and we conclude by Proposition \ref{aut_sec}. 

Now, consider the case $k < h-1$. Recall that $\mathcal{C}(n,m,h)$ is obtained from $\sec_h^{(k)}(\mathcal{S}^{n,m})$ by blow-ups centered at subvarieties of $\sec_h^{(k)}(\mathcal{S}^{n,m})$ that are stabilized by all $\phi\in\Aut(\sec_h^{(k)}(\mathcal{S}^{n,m}))$. Hence, $\phi\in\Aut(\sec_h^{(k)}(\mathcal{S}^{n,m}))$ lifts two an automorphism $\overline{\phi}$ of $\mathcal{C}(n,m,h)$, and we get a morphism of groups 
$$
\begin{array}{cccc}
\overline{\chi}:  & \Aut(\sec_{h}(\mathcal{S}^{n,m})) 
 & \longrightarrow & \Aut(\mathcal{C}(n,m,h))\\
      & \phi & \longmapsto & \overline{\phi}
\end{array}
$$
which again is an isomorphism. Finally, we conclude by the computation of $\Aut(\mathcal{C}(n,m,h))$ in the first part of the proof. 
\end{proof}

\section{Kontsevich spaces of conics and complete singular forms}\label{sec4}
An $n$-pointed rational pre-stable curve $(C,(x_{1},...,x_{n}))$ is a projective, connected, reduced curve with at most nodal singularities of arithmetic genus zero, with $n$ distinct and smooth marked points $x_1,...,x_n\in C$. We will refer to the marked and the singular points of $C$ as special points.

Let $X$ be a homogeneous variety. A map $(C,(x_{1},...,x_{n}),\alpha)$, where $\alpha:C\rightarrow X$ is a morphism from an $n$-pointed rational pre-stable curve to $X$ is stable if any component $E\cong\mathbb{P}^{1}$ of $C$ contracted by $\alpha$ contains at least three special points.

Now, let us fix a class $\beta\in H_2(X,\mathbb{Z})$. By \cite[Theorem 2]{FP} there exists a smooth, proper, and separated Deligne-Mumford stack $\overline{\mathcal{M}}_{0,n}(X,\beta)$ parametrizing isomorphism classes of stable maps $[C,(x_{1},...,x_{n}),\alpha]$ such that $\alpha_{*}[C] = \beta$.

Furthermore, by \cite[Corollary 1]{KP} the coarse moduli space $\overline{M}_{0,n}(X,\beta)$ associated to the stack $\overline{\mathcal{M}}_{0,n}(X,\beta)$ is a normal, irreducible, projective variety with at most finite quotient singularities of dimension
$$
\dim(\overline{M}_{0,n}(X,\beta)) = \dim(X)+\int_{\beta}c_1(T_X)+n-3.
$$
The variety $\overline{M}_{0,n}(X,\beta)$ is called the \textit{moduli space of stable maps}, or the \textit{Kontsevich moduli space} of stable maps of class $\beta$ from a rational pre-stable $n$-pointed curve to $X$. The boundary $\partial\overline{M}_{0,n}(X,\beta) = \overline{M}_{0,n}(X,\beta)\setminus M_{0,n}(X,\beta)$ is a simple normal crossing divisor in $\overline{M}_{0,n}(X,\beta)$ whose points parametrize isomorphism classes of stable maps $[C,(x_{1},...,x_{n}),\alpha]$ where $C$ is a reducible curve. When $X = \mathbb{P}^N$, we will write $\overline{M}_{0,n}(\mathbb{P}^N,d)$ for $\overline{M}_{0,n}(\mathbb{P}^N,d[L])$, where $L\subseteq\mathbb{P}^N$ is a line.

For details on moduli spaces parametrizing curves in projective spaces, and in particular conics, we refer to \cite[Section 8.4]{EH16}.

\subsection{Conics in $\mathbb{P}^n$}
Let $\overline{M}_{0,0}(\mathbb{P}^n,2)$ be the Kontsevich space of conics in $\P^n$. We will denote by $\Delta \subset \overline{M}_{0,0}(\mathbb{P}^n,2)$ the boundary divisor parametrizing maps with reducible domain, and by $\Gamma \subset \overline{M}_{0,0}(\mathbb{P}^n,2)$ the locus of maps of degree two onto a line. Note that $\Gamma$ is a $\mathbb{P}^2$-bundle over the Grassmannian $\mathbb{G}(1,n)$ parametrizing lines in $\mathbb{P}^n$. In $\overline{M}_{0,0}(\mathbb{P}^n,2)$ we consider the following divisor classes:
\begin{itemize}
\item[-] $\mathcal{H}$ of conics intersecting a fixed codimension two linear subspace of $\mathbb{P}^n$;
\item[-] $\mathcal{T}$ of conics which are tangent to a fixed hyperplane in $\mathbb{P}^n$;
\item[-] $D_{deg}$ of conics spanning a plane that intersects a fixed linear subspace of dimension $n-3$ in $\mathbb{P}^n$.
\end{itemize}

It is well-known that $\overline{M}_{0,0}(\mathbb{P}^2,2)$ is isomorphic to the space of complete conics $\mathcal{Q}(2,3)$ \cite[Section 0.4]{FP}. The following result generalizes this fact.

\begin{Proposition}\label{isoQ}
The Kontsevich space $\overline{M}_{0,0}(\mathbb{P}^n,2)$ is isomorphic to the blow-up $\sec_3^{(1)}(\mathcal{V}^n)$ of $\sec_3(\mathcal{V}^n)$ along $\mathcal{V}^n$.
\end{Proposition}
\begin{proof}
We may associate to a rank three quadric $Q\subset\mathbb{P}^n$ its dual conic $C_Q\subset \mathbb{P}^{n*}$. Conversely, given a smooth conic $C_Q\in \overline{M}_{0,0}(\mathbb{P}^n,2)$ we can consider the cone swept by the duals of the tangent lines of $C_Q$ and whose vertex is the dual of the plane spanned by $C_Q$. This yields a morphism
$$
\begin{array}{cccc}
\phi^{o}: & M_{0,0}(\mathbb{P}^n,2)& \longrightarrow & \sec_3(\mathcal{V}^n)\subset\mathbb{P}^{N_{+}}\\
      & C_Q & \longmapsto & Q.
\end{array}
$$ 
Consider the hyperplane $H = \{z_{0,0}=0\}\subset\mathbb{P}^{N_{+}}$. The points of $H\cap \sec_3(\mathcal{V}^n)$ correspond to the rank three quadrics $Q\subset\mathbb{P}^n$ passing through $p = [1:0:\dots :0]$. These quadric in turn correspond via the morphism $\phi^{o}$ to the conics $C_Q\subset \mathbb{P}^{n*}$ that are tangent the the hyperplane $H_p\subset \mathbb{P}^{n*}$ which is dual to $p\in\mathbb{P}^n$. Hence, $\phi^{o}$ is induced by the divisor class $\mathcal{T}$. Now, \cite[Theorem 1.2]{CHS09} yields that $\phi^{o}$ extends to a morphism 
$$
\phi:\overline{M}_{0,0}(\mathbb{P}^n,2)\rightarrow \sec_3(\mathcal{V}^n)
$$     
restricting to an isomorphism on $M_{0,0}(\mathbb{P}^n,2)$ and contracting the boundary divisor $\Delta$. 

Fix a rank two conic in $C_Q\subset \mathbb{P}^{n*}$. Up to an automorphism of $\mathbb{P}^{n*}$ we may assume that $C_Q = \{x_0 = \dots = x_{n-3} = x_{x-2}x_{n-1}=0\}$. Consider the family of smooth conics $C_{Q,t} = \{x_0 = \dots = x_{n-3} = x_{x-2}x_{n-1}-tx_n^2=0\}$, with $t\neq 0$, degenerating to $C_Q$. Then 
$$
\phi(C_{Q,t}) = \{x_n^2-4tx_{n-2}x_{n-1}=0\}
$$
where we keep denoting by $[x_0:\dots:x_n]$ the homogeneous coordinates on the dual projective space. Hence
$$
\phi(C_Q) = \lim_{t\mapsto 0}\phi(C_{Q,t}) = \{x_n^2 = 0\}
$$ 
and so $\Delta$ gets contracted onto the Veronese variety $\mathcal{V}^n\subset \sec_3(\mathcal{V}^n)$. Now, by \cite[Proposition 7.14]{Har77} $\phi$ yields a morphism 
$$
\psi:\overline{M}_{0,0}(\mathbb{P}^n,2)\rightarrow \sec_3^{(1)}(\mathcal{V}^n)
$$
mapping $\Delta$ onto $E_1^{\mathcal{Q}}$. Hence, $\psi$ restricts to a morphism $\psi_{|\Gamma}:\Gamma\rightarrow\sec_2^{(1)}(\mathcal{V}^n)$ associating to a double cover $\mathbb{P}^1\rightarrow L$ ramified at $p,q\in L$ the rank two quadric $H_p\cup H_q$, where $H_p,H_q$ are the hyperplanes dual to $p$ and $q$. Moreover, associating to a rank two quadric $H_1\cup H_2$ the $2$-to-$1$ cover $\mathbb{P}^1\rightarrow (H_1\cap H_2)^{*}$ ramified at $H_1^{*},H_2^{*}$ we get a birational inverse of $\psi_{|\Gamma}$. Note that $\psi_{|\Gamma}$ can not contract any divisor in $\Gamma$ since both $\Gamma$ and $\sec_2^{(1)}(\mathcal{V}^n)$ have Picard rank two. Furthermore, $\psi_{|\Gamma}$ can not contract any locus of codimension greater than one in $\Gamma$ either since $\sec_2^{(1)}(\mathcal{V}^n)$ is smooth. 

Hence, $\psi:\overline{M}_{0,0}(\mathbb{P}^n,2)\rightarrow \sec_3^{(1)}(\mathcal{V}^n)$ is a finite and birational morphism. Finally, since $\overline{M}_{0,0}(\mathbb{P}^n,2)$ and $\sec_3^{(1)}(\mathcal{V}^n)$ are normal \cite[Chapter 3, Section 9]{Mum99} yields that $\psi$ is an isomorphism.       
\end{proof}

As an application of Proposition \ref{isoQ} we have the following result. 

\begin{Proposition}\label{mcd_pn2}
The Kontsevich space $\overline{M}_{0,0}(\mathbb{P}^n,2)$ is a spherical variety with respect to the following $SL(n+1)$-action:
\stepcounter{thm}
\begin{equation}\label{actm22}
\begin{array}{cccc}
SL(n+1)\times \overline{M}_{0,0}(\mathbb{P}^n,2) & \longrightarrow & \overline{M}_{0,0}(\mathbb{P}^n,2)\\
(A,[C,\alpha]) & \longmapsto & [C,A\circ\alpha].
\end{array}
\end{equation}
The effective cone of $\overline{M}_{0,0}(\mathbb{P}^n,2)$ is generated by $\Delta$ and $D_{deg}$, and the nef cone of $\overline{M}_{0,0}(\mathbb{P}^n,2)$ is generated by $\mathcal{T}$ and $\mathcal{H}$. Furthermore, the following 
$$
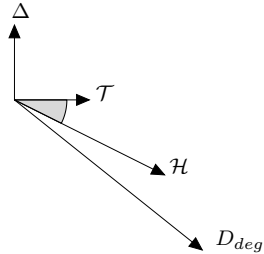
\begin{tikzpicture}[line cap=round,line join=round,>=triangle 45,x=1.0cm,y=1.0cm]
xmin=2.5,
xmax=7.5,
ymin=1.75,
ymax=5.301934941656083,
xtick={2.5,3.0,...,5.5},
ytick={2.0,2.5,...,5.0},]
\clip(2.1,1.65) rectangle (7.5,5.301934941656083);
\draw [->,line width=0.1pt] (3.,4.) -- (3.,5.);
\draw [->,line width=0.1pt] (3.,4.) -- (4.,4.);
\draw [->,line width=0.1pt] (3.,4.) -- (5.,3.);
\draw [->,line width=0.1pt] (3.,4.) -- (5.5,2.);
\draw [shift={(3.,4.)},line width=0.4pt,fill=black,fill opacity=0.15000000596046448]  (0,0) --  plot[domain=-0.46364760900080615:0.,variable=\t]({1.*0.6965067581669063*cos(\t r)+0.*0.6965067581669063*sin(\t r)},{0.*0.6965067581669063*cos(\t r)+1.*0.6965067581669063*sin(\t r)}) -- cycle ;
\begin{scriptsize}
\draw[color=black] (3.085808915158281,5.18) node {$\Delta$};
\draw[color=black] (4.2,4.1) node {$\mathcal{T}$};
\draw[color=black] (5.19,3.119262579937719) node {$\mathcal{H}$};
\draw[color=black] (6.0,2.118119471876817) node {$D_{deg}$};
\end{scriptsize}
\end{tikzpicture}
$$
is the Mori chamber decomposition of $\Eff(\overline{M}_{0,0}(\mathbb{P}^n,2))$, where $\mathcal{H}\sim 2\mathcal{T}-\Delta$ and $D_{deg}\sim 3\mathcal{T}-2\Delta$.
\end{Proposition}
\begin{proof}
The effective and the nef cone of $\overline{M}_{0,0}(\mathbb{P}^n,2)$ had already been computed in \cite[Theorem 1.5, Corollary 1.6]{CHS08} and \cite[Theorem 1.1]{CHS09} respectively.

The $SL(n+1)$-action on $\overline{M}_{0,0}(\mathbb{P}^n,2)$ in (\ref{actm22}) corresponds to the $SL(n+1)$-action on $\sec_3^{(1)}(\mathcal{V}^n)$ induced by (\ref{actquad}) via the isomorphism in Proposition \ref{isoQ}. Note that with respect to this action $\sec_3^{(1)}(\mathcal{V}^n)$ is spherical but not wonderful. However, we can deduce its boundary divisors and colors from those of $\mathcal{Q}(n,3)$ via the blow-down $\mathcal{Q}(n,3)\rightarrow\sec_3^{(1)}(\mathcal{V}^n)$ of $E_2^{\mathcal{Q}}$. Since boundary divisors and colors of $\sec_3^{(1)}(\mathcal{V}^n)$ lift to boundary divisors and colors of $\mathcal{Q}(n,3)$ by Proposition \ref{col_bou} we get that $E_1^{\mathcal{Q}}$ is the only boundary divisor of $\sec_3^{(1)}(\mathcal{V}^n)$, and that its colors are $D_1^{\mathcal{Q}},D_2^{\mathcal{Q}},D_3^{\mathcal{Q}}$, where we kept the same notation for divisors on $\mathcal{Q}(n,3)$ and $\sec_3^{(1)}(\mathcal{V}^n)$. Hence, arguing as in the proof of Proposition \ref{mcdq3} we get that $D_1^{\mathcal{Q}},D_2^{\mathcal{Q}}$ generate the nef cone of $\sec_3^{(1)}(\mathcal{V}^n)$, $D_3^{\mathcal{Q}},E_1^{\mathcal{Q}}$ generate it effective cone, and the Mori chamber decomposition of $\Eff(\sec_3^{(1)}(\mathcal{V}^n))$ has three chambers delimited respectively by the divisors $D_3^{\mathcal{Q}},D_2^{\mathcal{Q}}$, the divisors $D_2^{\mathcal{Q}},D_1^{\mathcal{Q}}$ and the divisors $D_1^{\mathcal{Q}},E_1^{\mathcal{Q}}$.

Now, by the proof of Proposition \ref{isoQ} we have that $E_1^{\mathcal{Q}}$ gets mapped to $\Delta$ by the isomorphism $\psi^{-1}: \sec_3^{(1)}(\mathcal{V}^n)\rightarrow\overline{M}_{0,0}(\mathbb{P}^n,2)$. Moreover, a straightforward computation shows that $\psi^{-1*}\mathcal{T} = D_1^{\mathcal{Q}}$, $\psi^{-1*}\mathcal{H} = D_2^{\mathcal{Q}}$ and $\psi^{-1*}D_{deg} = \frac{1}{2}D_3^{\mathcal{Q}}$. Finally, the statement follows from Lemma \ref{div_dec}, Proposition \ref{isoQ} and the description of the Mori chamber decomposition of $\Eff(\sec_3^{(1)}(\mathcal{V}^n))$ in the first part of the proof.
\end{proof}

\begin{Remark}\label{bir_n2}
We sum up the birational models of $\overline{M}_{0,0}(\mathbb{P}^n,2)$ in the following diagram: 
\[\begin{tikzcd}
	&& {\mathcal{Q}(n,3)} \\
	{\sec_3^{(1)}(\mathcal{V}^n)\cong \overline{M}_{0,0}(\mathbb{P}^n,2)} &&&& {\Hilb_{2}(\mathbb{P}^n)} \\
	&& {\Chow_2(\mathbb{P}^n)} \\
	{\sec_3(\mathcal{V}^n)} &&&& {\mathbb{G}(2,n)}
	\arrow[from=1-3, to=2-1]
	\arrow["{\mathcal{T}}"', from=2-1, to=4-1]
	\arrow["{\mathcal{H}}", from=2-1, to=3-3]
	\arrow[from=1-3, to=2-5]
	\arrow["\chi", dashed, from=2-1, to=2-5]
	\arrow[from=2-5, to=3-3]
	\arrow["{\widetilde{D}_{deg}}", from=2-5, to=4-5]
\end{tikzcd}\]
where $\Hilb_2(\mathbb{P}^n)$ and $\Chow_2(\mathbb{P}^n)$ are respectively the Hilbert scheme and the Chow variety of conics in $\mathbb{P}^n$, $\chi$ is the flip of $\Gamma\subset\overline{M}_{0,0}(\mathbb{P}^n,2)$, $\mathbb{G}(2,n)$ is the Grassmannians of planes in $\mathbb{P}^n$, and $\widetilde{D}_{deg}$ is the strict transform of $D_{deg}$ via $\chi$. The morphism induced by $\widetilde{D}_{deg}$ associates to a conic in $\Hilb_2(\mathbb{P})^n$ the unique plane of $\mathbb{P}^n$ in which it is contained. We would like to stress that the modular interpretation of the flip of $\overline{M}_{0,0}(\mathbb{P}^n,2)$ as a Hilbert scheme was well-know, see for instance \cite[Section 3]{Ki11}.
\end{Remark}

\subsection{Conics in $\mathbb{P}^n\times\mathbb{P}^m$}
Let $\mmnm$ be the Kontsevich space parametrizing conics in $\mathbb{P}^n \times \mathbb{P}^m$. Denote by $\pi: \overline{M}_{0,1}(\mathbb{P}^n \times \mathbb{P}^m,(1,1)) \rightarrow \mmnm$ the forgetful morphism, and by $ev: \overline{M}_{0,1}(\mathbb{P}^n \times \mathbb{P}^m,(1,1)) \rightarrow \mathbb{P}^n \times \mathbb{P}^m$ the evaluation morphism.

Let $H_n$ and $H_m$ be the hyperplane sections of $\mathbb{P}^n$ and $\mathbb{P}^m$ respectively, and $H_{n,m} \cong \mathbb{P}^{n-1} \times \mathbb{P}^{m-1} \subset \mathbb{P}^{n} \times \mathbb{P}^{m}$. Consider the divisors
$$\mathcal{K}^n:= \pi_*ev^*H^2_n,\: \mathcal{K}^m:= \pi_*ev^*H^2_m,\: \mathcal{K}^{n,m}:= \pi_*ev^*\mathcal{H}_{n,m}$$
and let $\Delta$ be the boundary divisor of maps with reducible domain.

By the proof of \cite[Lemma 1, Section 2.1]{Op05}, the Picard group of $\mmnm$ is generated by $\Delta, \mathcal{K}^n, \mathcal{K}^m$. In particular, since $H_1^2=0$, the Picard rank of $\mmnm$ is:
\stepcounter{thm}
\begin{equation}\label{picmmnm}
\rho(\mmnm) = \begin{cases} 
1 & \text{ if } n= m = 1;\\
2 & \text{ if } n=1, m\geq 2; \\
3 & \text{ if } n,m\geq 2.
\end{cases} 
\end{equation}

\begin{Proposition}\label{isoC}
The Kontsevich space $\mmnm$ is isomorphic to the space $\mathcal{C}(n,m,2)$ of rank two complete collineations on $\mathbb{P}^n\times\mathbb{P}^m$.
\end{Proposition}
\begin{proof}
First consider the case $n = m = 1$. We have that $\overline{M}_{0,0}(\mathbb{P}^1\times\mathbb{P}^1,(1,1))\cong\mathbb{P}^3$. Indeed, we may embed $\mathbb{P}^1\times\mathbb{P}^1$ in $\mathbb{P}^3$ as a smooth quadric $Q$, and the conics in $Q$ are in bijection with the hyperplanes in $\mathbb{P}^3$. 

Furthermore, $\mathcal{C}(1,1,2)\cong\mathbb{P}^3$ as well, and we may write down explicitly as isomorphism $\mathcal{C}(1,1,2)\rightarrow \overline{M}_{0,0}(\mathbb{P}^1\times\mathbb{P}^1,(1,1))$ as follows: write a point of $\mathcal{C}(1,1,2)$ as a $2\times 2$ matrix $Z$, fix homogeneous coordinates $([x_0:x_1],[y_0:y_1])$ on $\mathbb{P}^1 \times \mathbb{P}^1$, and associate to $Z$ the conic $C_Z=\{(x_0,x_1)\cdot Z\cdot (y_0,y_1)^t=0\} \subset\mathbb{P}^1 \times \mathbb{P}^1$.

Now, let $Z\in\sec_2(\mathcal{S}^{n,m}) \setminus \mathcal{S}^{n,m}$ an $(n+1) \times (m+1)$ matrix of rank two. The image of $Z$ yields a line $L_{Z}$ in $\mathbb{P}^n$, and the dual of the kernel of $Z$ gives a line $R_Z$ in $\mathbb{P}^{m*}$. Hence, we get a morphism
$$
\begin{array}{cccc}
\gamma^{o}: &\mathcal{C}(n,m,2)^{o}& \longrightarrow & \mathbb{G}(1,n)\times \mathbb{G}(1,m)\\
      & Z & \longmapsto & (L_Z,R_Z).
\end{array}
$$ 
The fiber of $\gamma^{o}$ over $(L_Z,R_Z)$ can be identified with the collineations on $L_Z\times R_Z$. To see this we argue as follows. Acting with $SL(n+1)\times SL(m+1)$ on $\mathbb{G}(1,n)\times \mathbb{G}(1,m)$ we may assume that $L_Z = \{x_2 = \dots = x_n = 0\}$ and $R_Z = \{y_2 = \dots = y_m = 0\}$. Hence, in $(\gamma^{o})^{-1}(L_Z,R_Z)$ we have the matrices annihilating the vectors $(0,0,1,0,\dots,0),\dots,(0,0,\dots,0,1)$ and whose image is generated by the vectors $(1,0,0,\dots,0),(0,1,0,\dots,0)$ that is matrices of the following form
$$
Z  = \left(\begin{array}{cc}
\overline{Z} & 0_{2,m-1} \\ 
0_{n-1,2} & 0_{n-1,m-1}
\end{array}\right),\: \text{with} \: 
\overline{Z}  = \left(\begin{array}{cc}
z_{0,0} & z_{0,1} \\ 
z_{1,0} & z_{1,1}
\end{array}\right).
$$
By the first part of the proof the collineations on $L_Z\times R_Z$ are in bijection with $\overline{M}_{0,0}(L_Z\times R_Z,(1,1))\subseteq \overline{M}_{0,0}(\mathbb{P}^n\times \mathbb{P}^m,(1,1))$. This yields an isomorphism
$$
\begin{array}{cccc}
\delta^{o}: &\mathcal{C}(n,m,2)^{o}& \longrightarrow & M_{0,0}(\mathbb{P}^n\times \mathbb{P}^m,(1,1))\\
      & Z & \longmapsto & C_Z.
\end{array}
$$ 
Now, consider the embedding $\overline{M}_{0,0}(\mathbb{P}^n\times \mathbb{P}^m,(1,1))\subset \overline{M}_{0,0}(\mathbb{P}^N,2)$. We will show that the inverse of $\delta^o$ is induced by the restriction to $M_{0,0}(\mathbb{P}^n\times \mathbb{P}^m,(1,1))$ of the divisor $\mathcal{T}$ on $\overline{M}_{0,0}(\mathbb{P}^N,2)$. Since $\mathcal{T}$ restricts on $\overline{M}_{0,0}(L_Z\times R_Z,(1,1))$ to the corresponding tangency divisor it is enough to show the claim for $\overline{M}_{0,0}(L_Z\times R_Z,(1,1))$. By the first part of the proof the correspondence between $\mathcal{C}(1,1,2)$ and $\overline{M}_{0,0}(\mathbb{P}^1\times \mathbb{P}^1,(1,1))$ is defined by mapping a matrix $Z = (z_{i,j})_{0\leq i,j\leq 1}$ to the divisor $C_Z = \{z_{0,0}x_0y_0+z_{0,1}x_0y_1+z_{1,0}x_1y_0+z_{1,1}x_1y_1\}\subset\mathbb{P}^1\times \mathbb{P}^1$ which in turn is mapped by the Segre embedding to the conic 
$$\overline{C}_Z = \{z_{0,0}X+z_{0,1}Y+z_{1,0}Z+z_{1,1}W = XW-YZ = 0\}\subset\mathbb{P}^3_{(X,Y,Z,W)}.$$ 
Now, considering the intersection of $\overline{C}_Z$ with the plane $\{W = 0\}$ we get the points $[z_{1,0}:0:-z_{0,0}:0]$ and $[z_{0,1}:-z_{0,0}:0:0]$. Therefore $\overline{C}_Z$ is tangent to $\{W = 0\}$ if and only if $z_{0,0} = 0$ that is if only if the matrix $Z$ lies on the hyperplane section $\{z_{0,0} = 0\}$ of $\mathbb{P}^N$. 

By \cite[Theorem 1.2]{CHS09} the divisor $\mathcal{T}$ is base point free and hence it restricts to a base point free divisor on $\overline{M}_{0,0}(\mathbb{P}^n\times \mathbb{P}^m,(1,1)$. Therefore, the inverse of $\delta^{o}$ is indeed a morphism 
$$
\begin{array}{cccc}
\eta: &\overline{M}_{0,0}(\mathbb{P}^n\times \mathbb{P}^m,(1,1))& \longrightarrow & \mathcal{C}(n,m,2)\\
      & C_Z & \longmapsto & Z
\end{array}
$$   
mapping the boundary divisor $\Delta$ to $E_1^{\mathcal{C}}$. Moreover, by Propositions \ref{picrank} and (\ref{picmmnm}) we get that $\eta$ does not contract any divisor. Finally, since $\mathcal{C}(n,m,2)$ is smooth we conclude, by \cite[Chapter 3, Section 9]{Mum99}, that $\eta$ is an isomorphism. 
\end{proof}

\begin{Remark}
Via the isomorphism 
$$\eta^{-1}:\mathcal{C}(n,m,2) \rightarrow \overline{M}_{0,0}(\mathbb{P}^n\times \mathbb{P}^m,(1,1))$$
we have 
$$\eta^{-1*}(\Delta) = E_1^{\mathcal{C}},\:\eta^{-1*}(\mathcal{K}^{n}) = H_1^{\mathcal{C}},\: \eta^{-1*}(\mathcal{K}^{m}) = H_2^{\mathcal{C}},\: \eta^{-1*}(\mathcal{K}^{n,m}) = D_1^{\mathcal{C}}.$$
These equalities together with Proposition \ref{mcd_C} give that for $n=1<m$, the Mori chamber decomposition of $\mmnm$ has two chambers delimited by $\Delta,\mathcal{K}^{n,m}$ and $\mathcal{K}^{n,m},\mathcal{K}^{m}$, while for $1<n \le m$ the Mori chamber decomposition of $\mmnm$ has three chambers delimited respectively by the divisors $\mathcal{K}^{n},\mathcal{K}^{m},\mathcal{K}^{n,m}$, the divisors $\mathcal{K}^{n},\mathcal{K}^{n,m},\Delta$ and the divisors $\mathcal{K}^{m},\mathcal{K}^{n,m},\Delta$.

Recall that a divisor of class $\mathcal{K}^{n}$ parametrizes stable maps $\alpha:\mathbb{P}^1\rightarrow\mathbb{P}^n\times\mathbb{P}^m$ intersecting a codimension two cycle of class $H_n^2$. These curves are mapped via the projection onto $\mathbb{P}^n$ to lines intersecting a fixed codimension two linear subspace of $\mathbb{P}^n$. Call $L_{\alpha}$ the line corresponding to the stable map $\alpha$. Note that these lines correspond in turn to a hyperplane section of the Grassmannian $\mathbb{G}(1,n)$ in its Pl\"ucker embedding. Hence, the semi-ample divisor $\mathcal{K}^{n}$ induces a morphism $\overline{M}_{0,0}(\mathbb{P}^n\times \mathbb{P}^m,(1,1))\rightarrow \mathbb{G}(1,n)$ associating to a map $[\mathbb{P}^1,\alpha]\in M_{0,0}(\mathbb{P}^n\times \mathbb{P}^m,(1,1))$ the line $L_{\alpha}$. Then by the proof of Proposition \ref{isoC} $H_1^{\mathcal{C}}$ yields a morphism $\mathcal{C}(n,m,2)\rightarrow\mathbb{G}(1,n)$ associating to a matrix $Z\in \mathcal{C}(n,m,2)^{o}$ the projectivization of its image.

Similarly, $\mathcal{K}^{m}$ induces a morphism $\overline{M}_{0,0}(\mathbb{P}^n\times \mathbb{P}^m,(1,1))\rightarrow \mathbb{G}(1,m)$ associating to a map $[\mathbb{P}^1,\alpha]\in M_{0,0}(\mathbb{P}^n\times \mathbb{P}^m,(1,1))$ the line $R_{\alpha}$ given by projecting the image of $\alpha$ to $\mathbb{P}^m$, and $H_2^{\mathcal{C}}$ yields a morphism $\mathcal{C}(n,m,2)\rightarrow\mathbb{G}(1,m)$ associating to a matrix $Z\in \mathcal{C}(n,m,2)^{o}$ the projectivization of the dual of its kernel.
\end{Remark}

\subsection{Conics in $\mathbb{G}(1,n)$}
Let $\mathbb{G}(1,n)$ be the Grassmannian of lines in $\mathbb{P}^n$. Following \cite[Section 2]{CC10} we describe divisor classes on $\overline{M}_{0,0}(\mathbb{G}(1,n),2)$. Fix projective subspaces $\Pi^{n-1},\Pi^{n-3}\subset\mathbb{P}^n$ of dimension $n-1$ and $n-3$, and consider the Schubert cycles
$$
\begin{array}{lll}
\sigma_{1,1}^{1,n} & = & \{W\in\mathbb{G}(1,n)\: | \: \dim(W\cap\Pi^{n-1})\geq 1\};\\ 
\sigma_{2}^{1,n} & = & \{W\in\mathbb{G}(1,n)\: | \: \dim(W\cap\Pi^{n-3})\geq 0\}.
\end{array} 
$$
Let $\pi:\overline{M}_{0,1}(\mathbb{G}(1,n),2)\rightarrow\overline{M}_{0,0}(\mathbb{G}(1,n),2)$ be the forgetful morphism and $ev:\overline{M}_{0,1}(\mathbb{G}(1,n),2)\rightarrow\mathbb{G}(1,n)$ the evaluation morphism. We define
$$H_{\sigma_{1,1}}^{1,n} = \pi_{*}ev^{*}\sigma_{1,1}, \: H_{\sigma_2}^{1,n} = \pi_{*}ev^{*}\sigma_2.$$
Furthermore, we will denote by $T^{1,n}$ the class of the divisor of conics that are tangent to a fixed hyperplane section of $\mathbb{G}(1,n)$.

Let $D_{deg}^{1,n}$ be the class of the divisor of maps $[C,\alpha]\in \overline{M}_{0,0}(\mathbb{G}(1,n),2)$ such that the projection of the span of the linear spaces parametrized by $\alpha(C)$ from a fixed subspace of dimension $n-4$ has dimension less than three.

Next we define the divisor class $D_{unb}^{1,n}$. A stable map $\alpha:\mathbb{P}^1\rightarrow\mathbb{G}(1,n)$ induces  a rank two subbundle $\mathcal{E}_{\alpha}\subset \mathcal{O}_{\mathbb{P}^1}\otimes\mathbb{C}^{n+1}$. We define $D_{unb}^{1,n}$ as the closure of the locus of maps $[\mathbb{P}^1,\alpha]\in \overline{M}_{0,0}(\mathbb{G}(1,n),2)$ such that $\mathcal{E}_{\alpha}\neq \mathcal{O}_{\mathbb{P}^1}(-1)^{\oplus 2}$.

Finally, we denote by $\Delta^{k,n}$ the boundary divisor parametrizing stable maps with reducible domain. 

\begin{Proposition}\label{mapgr}
There is a finite $2$-to-$1$ morphism
$$\varphi: \mmgu \longrightarrow \sec_4^{(2)}(\mathcal{V}^n)$$ 
mapping a stable map $[\mathbb{P}^1,\alpha]\in M_{0,0}(\mathbb{G}(1,n),2)$ to the rank four quadric $Q_C^{*} = \bigcup_{p\in Q_C}(T_pQ)^{*}\subset\mathbb{P}^{n*}$, where $Q_C = \bigcup_{[L]\in \alpha(\mathbb{P}^1)}L$.
\end{Proposition}
\begin{proof}
By \cite[Proposition 4.10, Theorem 5.1, Corollary 5.4]{CM17} there is a birational morphism $f:\mmgu\rightarrow\mathcal{T}_4^n$, contracting $D_{deg}^{1,n}$ and $\Delta^{1,n}$, where $\mathcal{T}_4^n$ is  the double symmetric determinantal locus of rank at most four constructed in \cite[Section 2.2]{HT15}. By \cite[Proposition 2.3]{HT15} there is a finite $2$-to-$1$ morphism $\rho:\mathcal{T}_4^n\rightarrow\sec_4(\mathcal{V}^n)$ branched along $\sec_3(\mathcal{V}^n)$. 

Now, consider the morphism $\rho\circ f:\mmgu\rightarrow \sec_4(\mathcal{V}^n)$. By \cite[Proposition 7.14]{Har77} there is a morphism $\varphi: \mmgu\rightarrow \sec_4^{(2)}(\mathcal{V}^n)$ such that $\pi\circ\varphi = \rho\circ f$, where $\pi: \sec_4^{(2)}(\mathcal{V}^n)\rightarrow\sec_4(\mathcal{V}^n)$ is the blow-down. 

Hence $\varphi$ is $2$-to-$1$ and by \cite[Theorem 1.1]{HT15} on $M_{0,0}(\mathbb{G}(1,n),2)$ it is defined by
$$
\begin{array}{lclc}
\varphi_{|M_{0,0}(\mathbb{G}(1,n),2)}: & M_{0,0}(\mathbb{G}(1,n),2) & \longrightarrow & \sec_4^{(2)}(\mathcal{V}^n)\\ 
 & [\mathbb{P}^1,\alpha] & \mapsto & Q_C^{*}
\end{array} 
$$
where $Q_C^{*} = \bigcup_{p\in Q_C}(T_pQ)^{*}\subset\mathbb{P}^{n*}$, and $Q_C = \bigcup_{[L]\in \alpha(\mathbb{P}^1)}L$. Note that $Q_C^{*}$ is indeed a quadric hypersurface of rank four, and since $Q_C$ can be constructed from either of its two rulings $\varphi_{|M_{0,0}(\mathbb{G}(1,n),2)}$ is $2$-to-$1$. 
\end{proof}

\begin{Remark}
For $n = 3$ the double cover in Proposition \ref{mapgr} had been constructed in \cite[Section 5]{Ce15}.
\end{Remark}

\begin{Remark}\label{CoCh}
As an application of Propositions \ref{MCD_G}, \ref{mapgr} we recover some results of \cite{CC10}. Indeed, on $\overline{M}_{0,0}(\mathbb{G}(1,n),2)$ there is an $SL(n+1)$-action given by 
$$
\begin{array}{cll}
SL(n+1) \times \overline{M}_{0,0}(\mathbb{G}(1,n),2) & \longrightarrow & \overline{M}_{0,0}(\mathbb{G}(1,n),2) \\ 
(M, [C, \alpha]) & \longmapsto & [C, \wedge^{2}M\circ\alpha] 
\end{array} 
$$
inducing on $\overline{M}_{0,0}(\mathbb{G}(1,n),2)$ a structure of spherical variety. 

Considering the subspace $H = \{x_4 = \dots = x_n=0\}\subset\mathbb{P}^n$ we get an embedding $i:\mathbb{G}(1,H) \hookrightarrow  \mathbb{G}(1,n)$ which in turn induces an embedding $j:\overline{M}_{0,0}(\mathbb{G}(1,3),2)\rightarrow\overline{M}_{0,0}(\mathbb{G}(1,n),2)$. Furthermore, the pull-back map $j^{*}:\Pic(\overline{M}_{0,0}(\mathbb{G}(1,n),2))\rightarrow\Pic(\overline{M}_{0,0}(\mathbb{G}(1,3),2)$ is an isomorphism. This reduces the study of the birational geometry of $\overline{M}_{0,0}(\mathbb{G}(1,n),2)$ to that of $\overline{M}_{0,0}(\mathbb{G}(1,3),2)$.

By Proposition \ref{MCD_G} and the $2$-to-$1$ morphism in Proposition \ref{mapgr} we get that the divisor classes $\Delta^{1,n}, D_{deg}^{1,n}, D_{unb}^{1,n}$ and the divisor classes $H_{\sigma_{1,1}}^{1,n},H_{\sigma_2}^{1,n},T^{1,n}$ are respectively the classes of the boundary divisors and the colors of the spherical variety $\overline{M}_{0,0}(\mathbb{G}(1,n),2)$. 

Furthermore, the divisors classes $D_{unb}^{1,n}, D_{deg}^{1,n}, \Delta^{1,n}$ generate the effective cone of $\overline{M}_{0,0}(\mathbb{G}(1,n),2)$. The Cox ring of $\overline{M}_{0,0}(\mathbb{G}(1,n),2)$ is generated by the global sections of the divisors $\Delta^{1,n}, D_{deg}^{1,n}, D_{unb}^{1,n}$ and $H_{\sigma_{1,1}}^{1,n},H_{\sigma_2}^{1,n},T^{1,n}$.

The nef cone of $\overline{M}_{0,0}(\mathbb{G}(1,n),2)$ is generated by $H_{\sigma_{1,1}}^{1,n}, H_{\sigma_2}^{1,n}, T^{1,n}$.  Moreover, the following is a $2$-dimensional section of the Mori chamber decomposition of $\Eff(\overline{M}_{0,0}(\mathbb{G}(1,n),2))$
$$
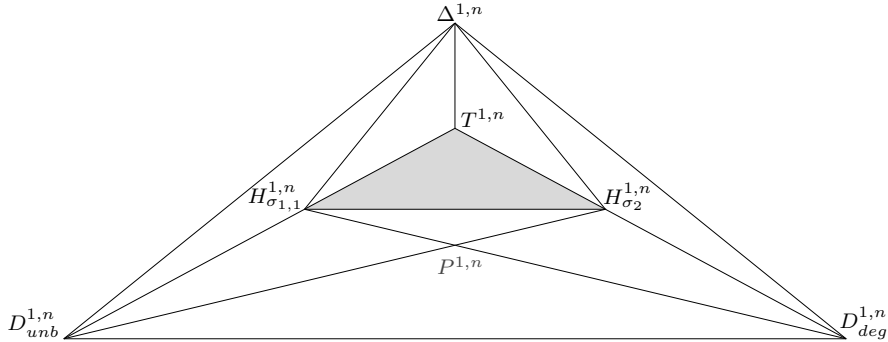
\begin{tikzpicture}[xscale=0.4,yscale=0.7][line cap=round,line join=round,>=triangle 45,x=1cm,y=1cm]\clip(-14.9,-0.21) rectangle (14.5,6.45);\fill[line width=0pt,fill=black,fill opacity=0.15] (-5.000432432432432,2.4614054054054053) -- (5.000432432432432,2.4614054054054053) -- (0,4) -- cycle;\fill[line width=0pt,color=wwwwww,fill=white,fill opacity=0.15] (-5.000432432432432,2.4614054054054053) -- (5.000432432432432,2.4614054054054053) -- (0,1.7776389756402244) -- cycle;\draw [line width=0.1pt] (-13,0)-- (13,0);\draw [line width=0.1pt] (13,0)-- (0,6);\draw [line width=0.1pt] (0,6)-- (-13,0);\draw [line width=0.1pt] (0,4)-- (-13,0);\draw [line width=0.1pt] (0,4)-- (13,0);\draw [line width=0.1pt] (-5.000432432432432,2.4614054054054053)-- (13,0);\draw [line width=0.1pt] (5.000432432432432,2.4614054054054053)-- (-13,0);\draw [line width=0.1pt] (-5.000432432432432,2.4614054054054053)-- (5.000432432432432,2.4614054054054053);\draw [line width=0.1pt] (-5.000432432432432,2.4614054054054053)-- (0,6);\draw [line width=0.1pt] (0,6)-- (5.000432432432432,2.4614054054054053);\draw [line width=0.1pt] (0,4)-- (0,6);\begin{scriptsize}\draw [fill=black] (-13,0) circle (0pt);\draw[color=black] (-14.0,0.3) node {$D_{unb}^{1,n}$};\draw [fill=black] (13,0) circle (0pt);\draw[color=black] (13.6,0.3) node {$D_{deg}^{1,n}$};\draw [fill=black] (0,6) circle (0pt);\draw[color=black] (0.18536585365853658,6.2) node {$\Delta^{1,n}$};\draw [fill=black] (0,4) circle (0pt);\draw[color=black] (0.95,4.2) node {$T^{1,n}$};\draw [fill=black] (-5.000432432432432,2.4614054054054053) circle (0pt);\draw[color=black] (-5.95,2.7) node {$H_{\sigma_{1,1}}^{1,n}$};\draw [fill=black] (5.000432432432432,2.4614054054054053) circle (0pt);\draw[color=black] (5.75,2.7) node {$H_{\sigma_2}^{1,n}$};\draw [fill=uuuuuu] (0,1.7776389756402244) circle (0pt);\draw[color=uuuuuu] (0.18536585365853658,1.4) node {$P^{1,n}$};\end{scriptsize}\end{tikzpicture}
$$
where $P^{1,n}\sim \frac{1}{4}(3H_{\sigma_{1,1}}^{1,n}+3H_{\sigma_2}^{1,n}-\Delta^{1,n})$, and $\Mov(\overline{M}_{0,0}(\mathbb{G}(1,n),2))$ is generated by $H_{\sigma_{1,1}}^{1,n}, H_{\sigma_2}^{1,n}, T^{1,n}, P^{1,n}$. 
\end{Remark}

We have the following result on the automorphisms of Kontsevich spaces of conics.

\begin{Corollary}\label{aut_M}
We have that 
$$
\Aut(\overline{M}_{0,0}(\mathbb{P}^n\times\mathbb{P}^m,(1,1))) \cong
\left\lbrace\begin{array}{ll}
PGL(n+1)\times PGL(m+1) & \text{if n}< \textit{m};\\ 
S_2 \ltimes (PGL(n+1)\times PGL(n+1)) & \text{if n} = \textit{m}\geq 2;
\end{array}\right.
$$
and $\Aut(\overline{M}_{0,0}(\mathbb{P}^1\times\mathbb{P}^1,(1,1)))\cong PGL(4)$.

Furthermore, $\Aut(\overline{M}_{0,0}(\mathbb{P}^n,2))\cong PGL(n+1)$ for $n\geq 3$, $\Aut(\overline{M}_{0,0}(\mathbb{P}^2,2))\cong PGL(3)\rtimes S_2$, and $\Aut(\overline{M}_{0,0}(\mathbb{P}^1,2))\cong PGL(3)$. 
\end{Corollary}
\begin{proof}
The first claim on $\Aut(\overline{M}_{0,0}(\mathbb{P}^n\times\mathbb{P}^m,(1,1)))$ follows from Proposition \ref{isoC} and Theorem \ref{autgr}. For the second claim recall that $\overline{M}_{0,0}(\mathbb{P}^1\times\mathbb{P}^1,(1,1))\cong\mathbb{P}^3$ since curves of bidegree $(1,1)$ in $\mathbb{P}^1\times\mathbb{P}^1$ are in bijection with the hyperplane sections of a smooth quadric surface in $\mathbb{P}^3$. 

The automorphism group of $\overline{M}_{0,0}(\mathbb{P}^n,2)$ for $n\geq 3$ follows from Proposition \ref{isoQ} and Theorem \ref{autgr}. The automorphism group of $\overline{M}_{0,0}(\mathbb{P}^2,2)$ has been computed in \cite[Remark 7.6]{Ma18a}. Finally, to get the claim on $\Aut(\overline{M}_{0,0}(\mathbb{P}^1,2))$ notice that $\overline{M}_{0,0}(\mathbb{P}^1,2)\cong\mathbb{P}^2$. Indeed, a $2$-to-$1$ morphism $\mathbb{P}^1\rightarrow\mathbb{P}^1$ is determined by its branch locus, and so $\overline{M}_{0,0}(\mathbb{P}^1,2)$ is isomorphic to $\mathbb{P}^1\times\mathbb{P}^1$ mod out by the involution switching the factors.
\end{proof}

Finally, we compute the automorphism group of $\mmgu$. Since the cases $n = 2$ has been covered in Corollary \ref{aut_M} we assume that $n\geq 3$. 

\begin{Proposition}\label{aut_MG}
The automorphism group of $\mmgu$ is given by 
$$
\Aut(\mmgu) \cong
\left\lbrace\begin{array}{ll}
S_2 \ltimes PGL(n+1) & \text{if n}>  3;\\ 
S_2 \ltimes (S_2 \ltimes PGL(n+1)) & \text{if n} = 3.
\end{array}\right.
$$
\end{Proposition}
\begin{proof}
First, consider the case $n = 3$. An automorphism of $\overline{M}_{0,0}(\mathbb{G}(1,3),2)$ must either preserve or switch the extremal rays $D_{unb}^{1,3}$ and $D_{deg}^{1,3}$. Indeed, there is an automorphism $\tau:\overline{M}_{0,0}(\mathbb{G}(1,3),2)\rightarrow \overline{M}_{0,0}(\mathbb{G}(1,3),2)$ switching them, namely the automorphism induced by the involution of $\mathbb{G}(1,3)$ given by projective duality. This yields a surjective morphism of groups
$$
\begin{array}{ccll}
\Psi:& \Aut(\overline{M}_{0,0}(\mathbb{G}(1,3),2)) & \longrightarrow & S_2\\ 
  & \varphi & \mapsto & \sigma_{\varphi} 
\end{array} 
$$
where $\sigma_{\varphi}$ is the permutation of the extremal rays of $\Eff(\overline{M}_{0,0}(\mathbb{G}(1,3),2))$ induced by $\varphi$. Now, assume that $\sigma_{\varphi}$ is trivial. Then $\varphi$ descends to an automorphism $\overline{\varphi}$ of the variety $\mathcal{T}_4^3$ in the proof of Proposition \ref{mapgr}. By \cite[Proposition 2.5 (3)]{HT15} $\mathcal{T}_4^3$ is Fano and the morphism $\rho:\mathcal{T}_4^3\rightarrow\sec_4(\mathcal{V}^{3})=\mathbb{P}^9$ in the proof of Proposition \ref{mapgr} is induced by a multiple of $-K_{\mathcal{T}_4^3}$. Hence, $\overline{\varphi}$ in turn descends to an automorphism of $\sec_4(\mathcal{V}^{3})=\mathbb{P}^9$ stabilizing the branch locus $\sec_3(\mathcal{V}^3)$. Since the group of automorphisms of $\mathbb{P}^9$ stabilizing $\sec_3(\mathcal{V}^3)$ is isomorphic to $PGL(4)$ we get an exact sequence
$$1\rightarrow S_2 \rightarrow \Aut(\mathcal{T}_4^3)\rightarrow PGL(4)\rightarrow 1.$$
Note that $PGL(4)$ acts on $\overline{M}_{0,0}(\mathbb{G}(1,3),2)$ and hence on $\mathcal{T}_4^3$. So the last morphism in the sequence has a section, and $\Aut(\mathcal{T}_4^3)\cong PGL(4)\rtimes S_2$.

Now, the morphism $\Psi$ yields the exact sequence
$$1\rightarrow \Aut(\mathcal{T}_4) \rightarrow \Aut(\overline{M}_{0,0}(\mathbb{G}(1,3),2))\rightarrow S_2\rightarrow 1$$
and since the last morphism in this sequence has a section we get the claim. 

When $n>3$ it is enough to argue as in the case $n = 3$ noticing that in this case $D_{unb}^{1,n}$ and $D_{deg}^{1,n}$ can not be switched and applying Proposition \ref{aut_sec}.
\end{proof}

\subsection{On the anti-canonical divisor}
In this last section we study the positivity of the anti-canonical divisor of the varieties in Propositions \ref{mcdq3}, \ref{mcd_C} and \ref{MCD_G}. Recall that a normal and $\mathbb{Q}$-factorial projective variety $X$ is
\begin{itemize}
\item[-] Fano if $-K_X$ is ample;
\item[-] weak Fano if $-K_X$ is nef and big;
\item[-] log Fano if there exists an effective divisor $D\subset X$ such that $-(K_X+D)$ is ample and the pair $(X,D)$ is Kawamata log terminal.
\end{itemize}
Clearly, Fano implies weak Fano which in turn implies log Fano. As a consequence of Kodaira's lemma \cite[Proposition 2.2.6]{La04} $X$ is log Fano if and only if there exists an effective divisor $D\subset X$ such that $-(K_X+D)$ is nef and big and the pair $(X,D)$ is Kawamata log terminal. Moreover, if $X$ and $Y$ are normal and $\mathbb{Q}$-factorial projective varieties which are isomorphic in codimension one then $X$ is log Fano if and only if $Y$ is so. We refer to \cite{GOST15} for further information on these notions. Finally, by \cite[Corollary 1.3.2]{BCHM10} if $X$ is log Fano then it is a Mori dream space.

\subsubsection{The anti-canonical divisor of $\mathcal{Q}(n,3)$}
If $n = 2$ then $\sec_3^{(1)}(\mathcal{V}^n)$ is the space of complete conics that is the blow-up of $\mathbb{P}^5$ along $\mathcal{V}^2$. So 
$$
-K_{\sec_3^{(1)}(\mathcal{V}^2)} = 6D_1^{\mathcal{Q}}-2E_1^{\mathcal{Q}} = 2(D_1^{\mathcal{Q}}+D_2^{\mathcal{Q}}).
$$
Assume $n\geq 3$. By \cite[Theorem 1.1]{dJS17} we have that
$$
-K_{\overline{M}_{0,0}(\mathbb{P}^n,2)} = \frac{3(n+1)}{4}\mathcal{H}-\frac{n-7}{4}\Delta
$$
and hence Proposition \ref{isoQ} yields 
$$
-K_{\sec_3^{(1)}(\mathcal{V}^n)} = \frac{3(n+1)}{2}D_1^{\mathcal{Q}}-(n-1)E_1^{\mathcal{Q}} = \frac{7-n}{2}D_1^{\mathcal{Q}}+(n-1)D_{2}^{\mathcal{Q}} = 3D_1^{\mathcal{Q}}+\frac{n-1}{2}D_3^{\mathcal{Q}}.
$$
Therefore, $\sec_3^{(1)}(\mathcal{V}^n)$ is Fano if and only if $1\leq n < 7$, weak Fano for $n = 7$ and log Fano for $n\geq 8$.

Now, note that by Proposition \ref{tcones} the tangent cone of $\sec_3^{(1)}(\mathcal{V}^n)$ at a point of $\sec_2^{(1)}(\mathcal{V}^n)\setminus (\sec_2^{(1)}(\mathcal{V}^n)\cap E_1^{\mathcal{Q}})$ is a cone with vertex of dimension $2n$ over $\mathcal{V}^{n-2}$. Hence, $\sec_3^{(1)}(\mathcal{V}^n)$ looks, locally around a point of $\sec_2^{(1)}(\mathcal{V}^n)\setminus (\sec_2^{(1)}(\mathcal{V}^n)\cap E_1^{\mathcal{Q}})$, as the weighted projective space $\mathbb{P}(1^{n-1},2^{2n+1})$. Therefore, $\sec_3^{(1)}(\mathcal{V}^n)$ has quotient singularities of type $\frac{1}{2}(1^{n-1})$ along $\sec_2^{(1)}(\mathcal{V}^n)\setminus (\sec_2^{(1)}(\mathcal{V}^n)\cap E_1^{\mathcal{Q}})$ and the discrepancy of the canonical divisor of $\mathcal{Q}(n,3)$ with respect to $E_2^{\mathcal{Q}}$ is $\frac{n-3}{2}$. Summing up we have
\stepcounter{thm}
\begin{equation}\label{KQn}
-K_{\mathcal{Q}(n,3)} = \frac{7-n}{2}D_1^{\mathcal{Q}}+(n-1)D_{2}^{\mathcal{Q}} - \frac{n-3}{2}E_2^{\mathcal{Q}} = 2D_1^{\mathcal{Q}} + 2D_2^{\mathcal{Q}} + \frac{n-3}{2}D_3^{\mathcal{Q}}.
\end{equation}
Hence, by Proposition \ref{mcdq3} and (\ref{KQn}) we get that  $\mathcal{Q}(n,3)$ if Fano for $n\geq 4$ and weak Fano for $n = 3$.

\subsubsection{The anti-canonical divisor of $\mathcal{C}(n,m,2)$}
The first two Chern classes of the tangent bundle $T_{\mathbb{P}^n\times \mathbb{P}^m}$ of $\mathbb{P}^n\times \mathbb{P}^m$ are given by 
$$
c_1 = (n+1)H_n + (m+1)H_m,\: c_2 = \binom{n+1}{2}H_n^2+(n+1)(m+1)H_nH_m+\binom{m+1}{2}H_m^2. 
$$
Hence, by \cite[Theorem 1.1]{dJS17} we have
\begin{scriptsize}
$$
-K_{\overline{M}_{0,0}(\mathbb{P}^n\times \mathbb{P}^m,(1,1))} = \frac{(n+1)(2n+m+3)}{2n+2m+4}\mathcal{K}^n + \frac{(n+1)(m+1)}{(n+m+2)}\mathcal{K}^{n,m} + \frac{(m+1)(2m+n+3)}{2n+2m+4}\mathcal{K}^m - \frac{nm-3n-3m-7}{2n+2m+4}\Delta
$$
\end{scriptsize}
and plugging in the relation $\Delta = 2\mathcal{K}^{n,m}-\mathcal{K}^{n}-\mathcal{K}^{m}$ from \cite[Section 2.2]{Op05} we get
\stepcounter{thm}
\begin{equation}\label{Knm}
-K_{\overline{M}_{0,0}(\mathbb{P}^n\times \mathbb{P}^m,(1,1))} = (n-1)\mathcal{K}^{n} + 4\mathcal{K}^{n,m}+(m-1)\mathcal{K}^{m}.
\end{equation}
As a consequence of Propositions \ref{mcd_C}, \ref{isoC} and (\ref{Knm}) we see that $\mathcal{C}(n,m,2)$ is Fano for all $n,m\geq 1$.

\subsubsection{The anti-canonical divisor of $\sec_4^{(2)}(\mathcal{V}^n)$}
By Proposition \ref{mapgr} there is a $2$-to-$1$ morphism 
$$\varphi:\overline{M}_{0,0}(\mathbb{G}(1,n),2)\rightarrow\sec_4^{(2)}(\mathcal{V}^n)$$ branched along $E_1^{\mathcal{Q}}$ and $\sec_3^{(2)}(\mathcal{V}^n)$. Note that $\sec_3^{(2)}(\mathcal{V}^n)$ is a divisor in $\sec_4^{(2)}(\mathcal{V}^n)$ if and only if $n = 3$. In this case $\sec_4^{(2)}(\mathcal{V}^3)$ is the space of complete quadrics of $\mathbb{P}^3$. So its anti-canonical divisor is given by 
$$
-K_{\sec_4^{(2)}(\mathcal{V}^3)} = 10D_1^{\mathcal{Q}} - 5E_1^{\mathcal{Q}} - 2E_2^{\mathcal{Q}}
$$  
and by Proposition \ref{MCD_G} $\sec_4^{(2)}(\mathcal{V}^3)$ is Fano. 

Assume that $n\geq 4$. Then $\sec_3^{(2)}(\mathcal{V}^n)$ has codimension greater than one in $\sec_4^{(2)}(\mathcal{V}^n)$ and so it does not play any role in the Riemann-Hurwitz formula relating the canonical divisors of $\overline{M}_{0,0}(\mathbb{G}(1,n),2)$ and $\sec_4^{(2)}(\mathcal{V}^n)$. By \cite[Remark 2.4]{CC10} we have that 
$$
-K_{\overline{M}_{0,0}(\mathbb{G}(1,n),2)} = \frac{11-n}{4}H^{1,n}_{\sigma
_{1,1}} + \frac{3n-1}{4}H^{1,n}_{\sigma
_{2}}+\frac{7-n}{4}\Delta^{1,n}.
$$
Write $-K_{\sec_4^{(2)}(\mathcal{V}^n)} = aD_1^{\mathcal{Q}} + b E_1^{\mathcal{Q}} + cE_2^{\mathcal{Q}}$. Since $\varphi^{*}D_1^{\mathcal{Q}} = H^{1,n}_{\sigma
_{1,1}}$, $\varphi^{*}D_2^{\mathcal{Q}} = T^{1,n}$, $\varphi^{*}D_3^{\mathcal{Q}} = H^{1,n}_{\sigma
_{2}}$, $\varphi^{*}E_1^{\mathcal{Q}} = 2D^{1,n}_{unb}$, $\varphi^{*}E_2^{\mathcal{Q}} = \Delta^{1,n}$
we have that 
$$
-K_{\overline{M}_{0,0}(\mathbb{G}(1,n),2)} = \varphi^{*}(-K_{\sec_4^{(2)}(\mathcal{V}^n)})-D^{1,n}_{unb} = \frac{4a+6b-3}{4} H^{1,n}_{\sigma
_{1,1}} + \frac{1-2b}{4} H^{1,n}_{\sigma
_{2}} + \frac{4c-2b+1}{4}\Delta^{1,n}
$$
where we used the relation $D^{1,n}_{unb} = \frac{1}{4}(3H^{1,n}_{\sigma
_{1,1}} - H^{1,n}_{\sigma
_{2}}-\Delta^{1,n})$ in \cite[Section 3]{CC10}. Finally,
\stepcounter{thm}
\begin{equation}\label{Ksec4}
-K_{\sec_4^{(2)}(\mathcal{V}^n)} = (2n+2)D_1^{\mathcal{Q}} - \frac{3n-2}{2}E_1^{\mathcal{Q}}-(n-2)E_2^{\mathcal{Q}} = 2D_1^{\mathcal{Q}}+\frac{6-n}{2}D_2^{\mathcal{Q}}+(n-3)D_3^{\mathcal{Q}}.
\end{equation}
By Proposition \ref{MCD_G} and (\ref{Ksec4}) we get that $\sec_4^{(2)}(\mathcal{V}^n)$ is Fano for $3\leq n\leq 5$ and weak Fano for $n = 6$. Furthermore, writing 
$$
-K_{\sec_4^{(2)}(\mathcal{V}^n)} = (8-n)D_1^{\mathcal{Q}}+3D_3^{\mathcal{Q}}+(n-6)P
$$
we see that $\sec_4^{(2)}(\mathcal{V}^n)$ is log Fano for $n\leq 8$.

\bibliographystyle{amsalpha}
\bibliography{Biblio}

\providecommand{\bysame}{\leavevmode\hbox to3em{\hrulefill}\thinspace}
\providecommand{\MR}{\relax\ifhmode\unskip\space\fi MR }
\providecommand{\MRhref}[2]{%
  \href{http://www.ams.org/mathscinet-getitem?mr=#1}{#2}
}
\providecommand{\href}[2]{#2}
\begin{thebibliography}{BCHM10}

\bibitem[ADHL15]{ADHL15}
I.~Arzhantsev, U.~Derenthal, J.~Hausen, and A.~Laface, \emph{Cox rings},
  Cambridge Studies in Advanced Mathematics, vol. 144, Cambridge University
  Press, Cambridge, 2015. \MR{3307753}

\bibitem[BCHM10]{BCHM10}
C.~Birkar, P.~Cascini, C.~D. Hacon, and J.~McKernan, \emph{Existence of minimal
  models for varieties of log general type}, J. Amer. Math. Soc. \textbf{23}
  (2010), no.~2, 405--468. \MR{2601039}

\bibitem[Bri89]{Br89}
M.~Brion, \emph{Groupe de {P}icard et nombres caract\'eristiques des
  vari\'et\'es sph\'eriques}, Duke Math. J. \textbf{58} (1989), no.~2,
  397--424. \MR{1016427}

\bibitem[Bri93]{Br93}
\bysame, \emph{Vari\'et\'es sph\'eriques et th\'eorie de {M}ori}, Duke Math. J.
  \textbf{72} (1993), no.~2, 369--404. \MR{1248677}

\bibitem[Bri07]{Br07}
\bysame, \emph{The total coordinate ring of a wonderful variety}, J. Algebra
  \textbf{313} (2007), no.~1, 61--99. \MR{2326138}

\bibitem[CC10]{CC10}
D.~Chen and I.~Coskun, \emph{Stable base locus decompositions of {K}ontsevich
  moduli spaces}, Michigan Math. J. \textbf{59} (2010), no.~2, 435--466.
  \MR{2677631}

\bibitem[CC11]{CC11}
\bysame, \emph{Towards {M}ori's program for the moduli space of stable maps},
  Amer. J. Math. \textbf{133} (2011), no.~5, 1389--1419, With an appendix by
  Charley Crissman. \MR{2843103}

\bibitem[Cha64]{Ch64}
M.~Chasles, \emph{Determination du nombre des sections conique qui doivent
  toucher cinq courbes donn\'ees d'ordre quelquonque, ou satisfaire \`a
  diverses autres conditions}, C.R. de l'Acad. de Sciences \textbf{58} (1864),
  222--226.

\bibitem[Che08]{Ch08}
D.~Chen, \emph{Mori's program for the {K}ontsevich moduli space
  {$\overline{\mathcal{M}}_{0,0}(\mathbb{P}^3,3)$}}, Int. Math. Res. Not. IMRN
  (2008), Art. ID rnn 067, 17. \MR{2439572}

\bibitem[CHS08]{CHS08}
I.~Coskun, J.~Harris, and J.~Starr, \emph{The effective cone of the
  {K}ontsevich moduli space}, Canad. Math. Bull. \textbf{51} (2008), no.~4,
  519--534. \MR{2462457}

\bibitem[CHS09]{CHS09}
\bysame, \emph{The ample cone of the {K}ontsevich moduli space}, Canad. J.
  Math. \textbf{61} (2009), no.~1, 109--123. \MR{2488451}

\bibitem[CM17]{CM17}
K.~Chung and H.~B. Moon, \emph{Mori's program for the moduli space of conics in
  {G}rassmannian}, Taiwanese J. Math. \textbf{21} (2017), no.~3, 621--652.
  \MR{3661384}

\bibitem[CS06]{CS06}
I.~Coskun and J.~Starr, \emph{Divisors on the space of maps to
  {G}rassmannians}, Int. Math. Res. Not. (2006), Art. ID 35273, 25.
  \MR{2264713}

\bibitem[DCP83]{DP83}
C.~{D}e Concini and C.~Procesi, \emph{Complete symmetric varieties}, Invariant
  theory ({M}ontecatini, 1982), Lecture Notes in Math., vol. 996, Springer,
  Berlin, 1983, pp.~1--44. \MR{718125}

\bibitem[Deb01]{De01}
O.~Debarre, \emph{Higher-dimensional algebraic geometry}, Universitext,
  Springer-Verlag, New York, 2001. \MR{1841091}

\bibitem[dJS17]{dJS17}
A.~J. de~Jong and J.~Starr, \emph{Divisor classes and the virtual canonical
  bundle for genus 0 maps}, Geometry over nonclosed fields, Simons Symp.,
  Springer, Cham, 2017, pp.~97--126. \MR{3644251}

\bibitem[EH16]{EH16}
D.~Eisenbud and J.~Harris, \emph{3264 and all that---a second course in
  algebraic geometry}, Cambridge University Press, Cambridge, 2016.
  \MR{3617981}

\bibitem[FP97]{FP}
W.~Fulton and R.~Pandharipande, \emph{Notes on stable maps and quantum
  cohomology}, Algebraic geometry---{S}anta {C}ruz 1995, Proc. Sympos. Pure
  Math., vol.~62, Amer. Math. Soc., Providence, RI, 1997, pp.~45--96.
  \MR{1492534}

\bibitem[Gia03]{Gi03}
G.~Z. Giambelli, \emph{Il problema della correlazione negli iperspazi}, Mem.
  Reale lnst. Lombardo \textbf{19} (1903), 155--194.

\bibitem[GOST15]{GOST15}
Y.~Gongyo, S.~Okawa, A.~Sannai, and S.~Takagi, \emph{Characterization of
  varieties of {F}ano type via singularities of {C}ox rings}, J. Algebraic
  Geom. \textbf{24} (2015), no.~1, 159--182. \MR{3275656}

\bibitem[Har77]{Har77}
R.~Hartshorne, \emph{Algebraic geometry}, Springer-Verlag, New York-Heidelberg,
  1977, Graduate Texts in Mathematics, No. 52. \MR{0463157}

\bibitem[Har95]{Ha95}
J.~Harris, \emph{Algebraic geometry}, Graduate Texts in Mathematics, vol. 133,
  Springer-Verlag, New York, 1995, A first course, Corrected reprint of the
  1992 original. \MR{1416564}

\bibitem[Hir75]{Hi75}
T.~A. Hirst, \emph{On {C}orrelation in {S}pace}, Proc. Lond. Math. Soc.
  \textbf{6} (1874/75), 7--9. \MR{1576746}

\bibitem[Hir77]{Hi77}
\bysame, \emph{Note on the {C}orrelation of {T}wo {P}lanes}, Proc. Lond. Math.
  Soc. \textbf{8} (1876/77), 262--273. \MR{1577537}

\bibitem[HK00]{HK00}
Y.~Hu and S.~Keel, \emph{Mori dream spaces and {GIT}}, Michigan Math. J.
  \textbf{48} (2000), 331--348, Dedicated to William Fulton on the occasion of
  his 60th birthday. \MR{1786494}

\bibitem[HT84]{HT84}
J.~Harris and L.~W. Tu, \emph{On symmetric and skew-symmetric determinantal
  varieties}, Topology \textbf{23} (1984), no.~1, 71--84. \MR{721453}

\bibitem[HT15]{HT15}
S.~Hosono and H.~Takagi, \emph{Geometry of symmetric determinantal loci},
  \url{https://arxiv.org/abs/1508.01995}, 2015.

\bibitem[Hue15]{Ce15}
C.~Lozano Huerta, \emph{Birational geometry of the space of complete quadrics},
  Int. Math. Res. Not. IMRN (2015), no.~23, 12563--12589. \MR{3431630}

\bibitem[Kie11]{Ki11}
Y.~H. Kiem, \emph{Birational geometry of moduli spaces of rational curves in
  projective varieties}, Higher dimensional algebraic geometry, RIMS
  K\^{o}ky\^{u}roku Bessatsu, B24, Res. Inst. Math. Sci. (RIMS), Kyoto, 2011,
  pp.~67--79. \MR{2809649}

\bibitem[KP01]{KP}
B.~Kim and R.~Pandharipande, \emph{The connectedness of the moduli space of
  maps to homogeneous spaces}, Symplectic geometry and mirror symmetry
  ({S}eoul, 2000), World Sci. Publ., River Edge, NJ, 2001, pp.~187--201.
  \MR{1882330}

\bibitem[KT88]{TK88}
S.~Kleiman and A.~Thorup, \emph{Complete bilinear forms}, Algebraic geometry
  ({S}undance, {UT}, 1986), Lecture Notes in Math., vol. 1311, Springer,
  Berlin, 1988, pp.~253--320. \MR{951650}

\bibitem[Laz04]{La04}
R.~Lazarsfeld, \emph{Positivity in algebraic geometry. {I}}, Ergebnisse der
  Mathematik und ihrer Grenzgebiete. 3. Folge. A Series of Modern Surveys in
  Mathematics [Results in Mathematics and Related Areas. 3rd Series. A Series
  of Modern Surveys in Mathematics], vol.~48, Springer-Verlag, Berlin, 2004,
  Classical setting: line bundles and linear series. \MR{2095471}

\bibitem[LLT89]{LLT89}
D.~Laksov, A.~Lascoux, and A.~Thorup, \emph{On {G}iambelli's theorem on
  complete correlations}, Acta Math. \textbf{162} (1989), no.~3-4, 143--199.
  \MR{989395}

\bibitem[Lun96]{Lu96}
D.~Luna, \emph{Toute vari\'et\'e magnifique est sph\'erique}, Transform. Groups
  \textbf{1} (1996), no.~3, 249--258. \MR{1417712}

\bibitem[Mas20a]{Ma18a}
A.~Massarenti, \emph{On the birational geometry of spaces of complete forms
  {I}: collineations and quadrics}, Proc. Lond. Math. Soc. (3) \textbf{121}
  (2020), no.~6, 1579--1618. \MR{4144371}

\bibitem[Mas20b]{Ma18b}
\bysame, \emph{On the birational geometry of spaces of complete forms {II}:
  {S}kew-forms}, J. Algebra \textbf{546} (2020), 178--200. \MR{4032731}

\bibitem[MP98]{MP98}
R.~MacPherson and C.~Procesi, \emph{Making conical compactifications
  wonderful}, Selecta Math. (N.S.) \textbf{4} (1998), no.~1, 125--139.
  \MR{1623714}

\bibitem[Mum99]{Mum99}
D.~Mumford, \emph{The red book of varieties and schemes}, expanded ed., Lecture
  Notes in Mathematics, vol. 1358, Springer-Verlag, Berlin, 1999, Includes the
  Michigan lectures (1974) on curves and their Jacobians, With contributions by
  Enrico Arbarello. \MR{1748380}

\bibitem[Oka16]{Ok16}
S.~Okawa, \emph{On images of {M}ori dream spaces}, Math. Ann. \textbf{364}
  (2016), no.~3-4, 1315--1342. \MR{3466868}

\bibitem[Opr05]{Op05}
D.~Oprea, \emph{Divisors on the moduli spaces of stable maps to flag varieties
  and reconstruction}, J. Reine Angew. Math. \textbf{586} (2005), 169--205.
  \MR{2180604}

\bibitem[Per14]{Pe14}
N.~Perrin, \emph{On the geometry of spherical varieties}, Transform. Groups
  \textbf{19} (2014), no.~1, 171--223. \MR{3177371}

\bibitem[Sch86]{Sc86}
H.~Schubert, \emph{Die {$n$}-dimensionalen {V}erallgemeinerungen der
  fundamentalen {A}nzahlen unseres {R}aums}, Math. Ann. \textbf{26} (1886),
  no.~1, 26--51. \MR{1510326}

\bibitem[Sch10]{Sch99}
H.~Schoutens, \emph{The use of ultraproducts in commutative algebra}, Lecture
  Notes in Mathematics, vol. 1999, Springer-Verlag, Berlin, 2010. \MR{2676525}

\bibitem[Seg84]{Se84}
C.~Segre, \emph{Studio sulle quadriche in uno spazio lineare ad un numero
  qualunque di dimensioni}, Memorie dell'Acc. dei Lincei \textbf{19} (1883/84),
  127--148.

\bibitem[Sem48]{Se48}
J.~G. Semple, \emph{On complete quadrics}, J. London Math. Soc. \textbf{23}
  (1948), 258--267. \MR{0028605}

\bibitem[Sem51]{Se51}
\bysame, \emph{The variety whose points represent complete collineations of
  {$S_r$} on {$S'_r$}}, Univ. Roma. Ist. Naz. Alta Mat. Rend. Mat. e Appl. (5)
  \textbf{10} (1951), 201--208. \MR{0048847}

\bibitem[Sem52]{Se52}
\bysame, \emph{On complete quadrics. {II}}, J. London Math. Soc. \textbf{27}
  (1952), 280--287. \MR{0048846}

\bibitem[Tha99]{Tha99}
M.~Thaddeus, \emph{Complete collineations revisited}, Math. Ann. \textbf{315}
  (1999), no.~3, 469--495. \MR{1725990}

\bibitem[Tyr56]{Ty56}
J.~A. Tyrrell, \emph{Complete quadrics and collineations in {$S_n$}},
  Mathematika \textbf{3} (1956), 69--79. \MR{0080352}

\bibitem[Vai82]{Va82}
I.~Vainsencher, \emph{Schubert calculus for complete quadrics}, Enumerative
  geometry and classical algebraic geometry ({N}ice, 1981), Progr. Math.,
  vol.~24, Birkh\"auser, Boston, Mass., 1982, pp.~199--235. \MR{685770}

\bibitem[Vai84]{Va84}
\bysame, \emph{Complete collineations and blowing up determinantal ideals},
  Math. Ann. \textbf{267} (1984), no.~3, 417--432. \MR{738261}

\end{thebibliography}

\end{document}